\documentclass[11pt]{article}

\def\UseSection{
        \numberwithin{equation}{section}
        \newtheorem{theorem}    {Theorem}[section]
        \DefineTheorems 
}
\usepackage[textwidth=500pt,textheight=660pt,centering]{geometry} 

\usepackage{amsfonts}
\usepackage{amsmath,amssymb,amsthm}
\usepackage{appendix}
\usepackage{bbm} 
\usepackage{amsbsy}
\usepackage{enumerate}
\usepackage{cite}

\usepackage[bookmarks, colorlinks, breaklinks,
pdftitle={},pdfauthor={}]{hyperref}

\hypersetup{
  linkcolor=black,
  citecolor=black,
  filecolor=black,
  urlcolor=black}

\usepackage{verbatim}
\usepackage{tikz}
\usepackage{color}

\usepackage{extarrows}

\newcommand{\black}{\black}

\usepackage[font=small,labelfont=bf]{caption}

\usepackage{setspace}
\numberwithin{equation}{section}

\usetikzlibrary{calc,decorations.markings}
\usetikzlibrary{decorations.pathmorphing}
\usetikzlibrary{decorations.pathreplacing}
\usetikzlibrary{patterns}
\usetikzlibrary{arrows}

\newcommand{\bb}[1]{\mathbb{#1}}

\newcommand{\1}{\mathbbm{1}}

\newcommand{\D}[0]{\text{d}}
\newcommand{\conn}{\leftrightarrow}

\newcommand{\V}[1]{\boldsymbol{#1}}
\newcommand{\blank}[1]{}

\newcommand{\B}{\bb B}
\newcommand{\E}{\bb E}
\newcommand{\R}{\bb R}
\newcommand{\Z}{\bb Z}

\newcommand{\G}{\bb G}
\newcommand{\N}{\bb N}
\newcommand{\T}{\bb T}
\renewcommand{\P}{\bb P}
\newcommand{\Q}{\bb Q}
\renewcommand{\V}{\bb V}

\newcommand{\Cov}{\operatorname{Cov}}
\newcommand{\CoVr}{\operatorname{CoVr}}
\newcommand{\lrDini}[1]{\left(\frac{d}{d #1}\right)_{\hspace{-0.2em}+}\!}

\makeatletter
\pgfdeclareshape{slash underlined}
{
  \inheritsavedanchors[from=rectangle] 
  \inheritanchorborder[from=rectangle]
  \inheritanchor[from=rectangle]{north}
  \inheritanchor[from=rectangle]{north west}
  \inheritanchor[from=rectangle]{north east}
  \inheritanchor[from=rectangle]{center}
  \inheritanchor[from=rectangle]{west}
  \inheritanchor[from=rectangle]{east}
  \inheritanchor[from=rectangle]{mid}
  \inheritanchor[from=rectangle]{mid west}
  \inheritanchor[from=rectangle]{mid east}
  \inheritanchor[from=rectangle]{base}
  \inheritanchor[from=rectangle]{base west}
  \inheritanchor[from=rectangle]{base east}
  \inheritanchor[from=rectangle]{south}
  \inheritanchor[from=rectangle]{south west}
  \inheritanchor[from=rectangle]{south east}
  \inheritanchorborder[from=rectangle]
  \foregroundpath{
    \southwest \pgf@xa=\pgf@x \pgf@ya=\pgf@y
    \northeast \pgf@xb=\pgf@x \pgf@yb=\pgf@y
    \pgf@xc=\pgf@xa
    \advance\pgf@xc by .5\pgf@xb
    \pgf@yc=\pgf@ya
    \advance\pgf@xc by -1.3pt
    \advance\pgf@yc by -1.8pt
    \pgfpathmoveto{\pgfqpoint{\pgf@xc}{\pgf@yc}}
    \advance\pgf@xc by  2.6pt
    \advance\pgf@yc by  3.6pt
    \pgfpathlineto{\pgfqpoint{\pgf@xc}{\pgf@yc}}
    \pgfpathmoveto{\pgfqpoint{\pgf@xa}{\pgf@ya}}
    \pgfpathlineto{\pgfqpoint{\pgf@xb}{\pgf@ya}}
 }
}
\tikzset{nomorepostaction/.code=\let\tikz@postactions\pgfutil@empty}

\newcommand\nxleftrightarrow[2][]{%
  \mathrel{\tikz[baseline=-.7ex] \path node[slash underlined,draw,<->,anchor=south] {\(\scriptstyle #2\)} node[anchor=north] {\(\scriptstyle #1\)};}}
\makeatother

\newcommand{\cF}{\mathcal F}
\newcommand{\cG}{\mathcal G}

\newcommand{\cP}{\mathcal P}

\newcommand{\sA}{\mathscr A}
\newcommand{\sB}{\mathscr B}
\newcommand{\sC}{\mathscr C}

\newcommand{\bbB}{\mathbb B}

\newcommand{\bbT}{\mathbb T}

\newcommand{\veee}[1]{\langle #1 \rangle}
\newcommand{\xvee}{\veee{x}}

\usepackage{mathrsfs}

\newcommand{\eps}{\varepsilon}

\usepackage{cleveref}
\usepackage{caption}
\usepackage{thmtools}
\usepackage{mathrsfs}

\crefname{theorem}{Theorem}{Theorems}
\crefname{thm}{Theorem}{Theorems}
\crefname{lemma}{Lemma}{Lemmas}
\crefname{claim}{Claim}{Claims}
\crefname{lem}{Lemma}{Lemmas}
\crefname{remark}{Remark}{Remarks}
\crefname{prop}{Proposition}{Propositions}
\crefname{proposition}{Proposition}{Propositions}
\crefname{defn}{Definition}{Definitions}
\crefname{definition}{Definition}{Definitions}
\crefname{corollary}{Corollary}{Corollaries}
\crefname{conjecture}{Conjecture}{Conjectures}
\crefname{question}{Question}{Questions}
\crefname{chapter}{Chapter}{Chapters}
\crefname{section}{Section}{Sections}
\crefname{part}{Part}{Parts}
\crefname{figure}{Figure}{Figures}

\newcommand{\nnb}	{\nonumber \\}
\newcommand{\bubble}{{\sf B}}
\renewcommand{\triangle}{{\sf T}}

\newtheorem*{rk-non}  {Remark}
\def\DefineTheorems{
	\newtheorem{lemma}      [theorem] {Lemma}
	\newtheorem{cor}        [theorem] {Corollary}
	
	\newtheorem{prop}        [theorem] {Proposition}
	\theoremstyle{definition}
	\newtheorem{defn}       [theorem] {Definition}
	\newtheorem{rk}       [theorem] {Remark}

}

\UseSection
\title{High-dimensional near-critical percolation  \\
and the torus plateau}
 \author{
    Tom Hutchcroft\thanks{The Division of Physics, Mathematics and Astronomy, California Institute of Technology, Pasadena, CA 91125, USA
    \url{https://orcid.org/0000-0003-0061-593X}, {\tt t.hutchcroft@caltech.edu}} \and
    Emmanuel Michta\thanks{Department of Mathematics,
     University of British Columbia,
     Vancouver, BC, Canada V6T 1Z2.
     Michta: \url{https://orcid.org/0000-0001-7222-0422}, {\tt michta@math.ubc.ca}.
     Slade:  \url{https://orcid.org/0000-0001-9389-9497}, {\tt slade@math.ubc.ca}.}
    \and
   Gordon Slade$^\dagger$}

\begin{document}

\date{\vspace{-5ex}} 

\maketitle

\begin{abstract}
We consider percolation on $\mathbb{Z}^d$ and on the $d$-dimensional discrete torus, in dimensions
$d \ge 11$ for the nearest-neighbour model and in dimensions $d>6$ for spread-out models.
For $\mathbb{Z}^d$, we
 employ a wide range of techniques and previous results to
prove that there exist positive constants $c$ and $C$ such that the slightly subcritical two-point function and one-arm probabilities satisfy
\[
\P_{p_c-\varepsilon}(0 \leftrightarrow x) \leq \frac{C}{\|x\|^{d-2}} e^{-c\varepsilon^{1/2} \|x\|}
\quad \text{ and } \quad
\frac{c}{r^{2}} e^{-C \varepsilon^{1/2}r} \leq
\P_{p_c-\varepsilon}\Bigl(0 \leftrightarrow \partial [-r,r]^d \Bigr)
\leq \frac{C}{r^2} e^{-c \varepsilon^{1/2}r}.
\]
Using this, we prove that
 throughout the critical window
the torus two-point function has a ``plateau,''
meaning that it decays for small $x$ as $\|x\|^{-(d-2)}$ but for large $x$ is essentially constant
and of order $V^{-2/3}$ where $V$ is the volume of the torus.  The plateau for the two-point
function leads immediately to a proof of the torus triangle condition, which is known
to have many implications for the critical behaviour on the torus, and also leads to a proof
that the critical values on the torus and on $\mathbb{Z}^d$ are separated by a multiple of
$V^{-1/3}$.
The torus triangle condition and the size of the separation of critical points have been
proved previously, but our proofs are different and are direct consequences
of the bound on the $\mathbb{Z}^d$ two-point function.  In particular,
we use results derived from the lace expansion on $\mathbb{Z}^d$, but in contrast
to previous work on high-dimensional torus percolation we do not need or use a separate
torus lace expansion.
\end{abstract}



\tableofcontents

\section{Introduction and results}

\subsection{Introduction}

Percolation on $\mathbb{Z}^d$
has been intensively studied by mathematicians and physicists since the 1950s
as a fundamental model of a phase transition.  Of particular interest is the universal
critical behaviour in the vicinity of the critical value $p_c$.  From a mathematical
perspective, the critical behaviour has been established for certain 2-dimensional models
using the breakthroughs enabled by conformal invariance and the Schramm--Loewner
evolution \cite{Smir01CR,SW01},
and for a wide class of models above the upper critical dimension $d=6$ using the lace expansion \cite{HS90a,HH17book}.
The critical behaviour in
intermediate dimensions $d=3,4,5,6$ remains a major challenge for probability theory,
which at present appears not to have adequate tools even to approach the problem.
Considerable progress has been also made in the understanding of
the finite-size scaling associated with critical percolation on a high-dimensional discrete torus.
In this paper, we consider percolation in dimensions $d>6$, both on $\mathbb{Z}^d$ and on
the torus.

The role of $d=6$ as the upper critical dimension for percolation was first pointed out by
Toulouse~\cite{Toul74}.  The meaning of ``upper critical dimension''
is that the critical exponents for percolation
on $\Z^d$ in dimensions $d>6$ are predicted to be the same as for percolation on a tree
(known as mean-field theory),
whereas for $d<6$ they are not.  Critical exponents for $d=6$ are predicted to have
logarithmic corrections to mean-field behaviour \cite{EGG78}.  Various one-sided mean-field bounds
for critical exponents, such as $\gamma \ge 1$, $\beta \le 1$, and $\delta \ge 2$,
 have been proven to hold in all dimensions \cite{AN84,BA91}.
In addition, results implying that mean-field behaviour cannot apply in dimensions $d<6$
have been obtained in \cite{CC87,Tasa87} (see also \cite{BCKS99}).
In an important paper in 1984, Aizenman and Newman \cite{AN84} identified a
condition predicted to be valid for $d>6$---the
\emph{triangle condition}---as a sufficient condition
for $\gamma=1$, which is mean-field behaviour for the
expected cluster size (also called the susceptibility).  Then Barsky and Aizenman \cite{BA91} proved
that the triangle condition also implies that $\beta=1$ (exponent for the percolation
probability) and $\delta=2$ (exponent for the magnetisation). See also
the recent paper
\cite{Hutc22} for alternative proofs of these results.

In 1990, Hara and Slade \cite{HS90a} derived their lace expansion for bond percolation and used it to
verify the triangle condition for the nearest-neighbour model
in sufficiently large dimensions ($d \ge 19$ is large enough
\cite{HS94}).
Later, Fitzner and van der Hofstad \cite{FH17} extended this to all $d \ge 11$.
An extension to $d>6$ has not yet been possible, and seems to be impossible without
the introduction of some significant new idea, due to the fact that convergence
of the lace expansion is proved using a small parameter which is closely related to the
triangle diagram and which is not believed to be small in dimensions close to but above $6$.
On the other hand, since the critical exponents are predicted to be \emph{universal}, meaning
that they take the same values for any symmetric
short-range
model in a given dimension $d$,
it is natural to introduce models with a parameter that \emph{can} be taken to be small in any
fixed dimension $d>6$.  This was accomplished in \cite{HS90a}, where the triangle condition
was proved for a wide variety of sufficiently spread-out models in any dimension $d>6$.
A basic example of a spread-out model is a finite-range model
of bond percolation on $\Z^d$ with long bonds, not just
nearest-neighbour bonds, and
the reciprocal of the degree
provides a small parameter
for convergence of the lace expansion in any dimension $d>6$. Related results for long-range
 models have also been established in \cite{CS14,HHS08}.

Over the last thirty years, a large literature on high-dimensional percolation has emerged,
using the convergence of the lace expansion as a starting point, typically both
for sufficiently spread-out models
in dimensions $d >6$ and for the nearest-neighbour model in large enough dimensions.
Reviews can be found in \cite{HH17book,Slad06}.  In particular,
Hara proved the square root decay of the mass (inverse correlation length)
\cite{Hara90};
Hara, van der Hofstad and Slade \cite{HHS03} and Hara \cite{Hara08} proved that
the critical two-point function has the Gaussian decay $|x|^{-(d-2)}$;
Kozma and Nachmias
\cite{KN11} proved the mean-field behaviour $r^{-2}$
for the one-arm exponent; and Chatterjee and Hanson
\cite{CH20} identified the decay of the critical two-point function in a half-space.
We use these results to prove our main results for high-dimensional percolation on $\Z^d$. Following the methodology of \cite{Hutc20},
we also use the OSSS theory of decision trees \cite{OSSS05}, whose application to
statistical mechanical models was pioneered by Duminil-Copin, Raoufi, and Tassion \cite{DRT19}, to obtain a new differential
inequality which facilitates the transfer of certain estimates at the critical point to
estimates at nearby subcritical points.

Our results for $\Z^d$ consist of an upper bound of the form $|x|^{-(d-2)}\exp [-c|p-p_c|^{1/2}|x|]$ for the
slightly subcritical two-point function, and upper and lower bounds of the
form $r^{-2}\exp [-c|p-p_c|^{1/2}r]$ for the (extrinsic) one-arm probability.
We stress for the avoidance of doubt that the inclusion of these sharp exponential factors for $p<p_c$, with the square root
in the exponent, requires substantial new ideas and is not a minor extension of the previous results.

In a separate line of research initiated by Borgs, Chayes, van der Hofstad, Slade, and Spencer in \cite{BCHSS05a,BCHSS05b}, the critical
behaviour of percolation on a discrete $d$-dimensional torus has been studied in depth,
both for the nearest-neighbour model with $d$ sufficiently large and for
sufficiently spread-out models in dimensions $d>6$.  There is a triangle
condition for the torus (and indeed for general high-dimensional transient graphs)
which implies that percolation on the torus behaves in many respects like
the Erd\H{o}s--R\'enyi random graph.
In particular, the notion of a critical point which is valid for $\Z^d$ is replaced
by the notion of a critical scaling window of $p$ values.
These ideas are developed in \cite{BCHSS05a,HHI07,HHII11,HH17book,HS14}, and are based on
the verification of the triangle condition in high-dimensions via a separate lace
expansion on the torus as opposed to on $\Z^d$ \cite{BCHSS05b}.

Our first result for the torus is a proof that the torus two-point function has a ``plateau.''
The plateau refers to the fact that the torus two-point function within and slightly
below the critical window decays for small $x$ like its $\Z^d$ counterpart before levelling off at a constant value for large $x$.
  Related plateaux have been proven to exist
for simple random walk (the lattice Green function) for $d>2$ \cite{Slad20_wsaw,ZGDG20,ZGFDG18}, for weakly self-avoiding walk for $d>4$
\cite{Slad20_wsaw}, and partially for the Ising model for $d>4$ \cite{Papa06}.
As we show,
 the plateau for the torus percolation two-point function is highly effective
for the analysis of torus percolation (a similar situation applies for weakly
self-avoiding walk on a torus for $d>4$ \cite{MS22}).
In particular, it directly gives
a proof of the torus triangle condition,
a proof that throughout the critical window
the torus susceptibility is of the order of the cube root of the torus
volume, and a proof that the $\Z^d$ critical
value lies in the critical window for the torus.  The triangle condition was proved
previously via a separate lace expansion on the torus \cite{BCHSS05b} which we do not
need,
the behaviour of the torus susceptibility was obtained previously in \cite{BCHSS05a},
while the verification that the $\Z^d$ critical value lies in the window was the
main topic of \cite{HHI07,HHII11}.  Our work establishes these results directly by applying
results on $\Z^d$ rather than via a separate torus lace expansion.

\subsection{The models}
\label{sec:models}

Let $\G=(\V,\B)$ be a finite or infinite graph with vertex set $\V$ and
edge (bond) set $\B$.
Given $p \in [0,1]$, we consider independent and identically distributed
Bernoulli random variables associated to each bond $b \in \B$,
taking the value ``open'' with probability $p$ and the value ``closed'' with probability $1-p$.  We denote the probability of an event $E$ by $\P_p(E)$ and the expectation of
a random variable $X$ by $\E_pX$.

We consider four different choices of $\G$:
\begin{enumerate}[(i)]
\item
Nearest-neighbour model on $\Z^d$:  $\V=\Z^d$ and $\B$ consists of all pairs $\{x,y\}$
with $\|x-y\|_1=1$.  We assume that $d \ge 11$.
\item
Spread-out model on $\Z^d$:  $\V=\Z^d$ and $\B$ consists of all pairs $\{x,y\}$
with $\|x-y\|_1\le L$, for some (large) fixed $L > 1$.  We assume that $d>6$
and $L$ is sufficiently large depending on $d$.
\item
Nearest-neighbour model on the torus $\T_r^d$:
$\V = (\Z/r\Z)^d$ for (large) period $r > 2$
 and $\B$ consists of all pairs $\|x-y\|_1=1$ with addition mod $r$.
We assume that $d \ge 11$.
We write $V=r^d$ for the \emph{volume} of the torus and are
interested in the limit $r \to \infty$.
\item
Spread-out model on the torus $\T_r^d$:
$\V = (\Z/r\Z)^d$ for (large) period $r > 2L$
with (large) fixed $L > 1$
and $\B$ consists of all pairs $\|x-y\|_1\le L$ with addition mod $r$.
We assume that $d >6$ and $L$ is sufficiently large depending on $d$.
\end{enumerate}

\noindent
\textbf{Notation:}
We set $\N = \{1,2, \ldots\}$.
 We use $c,C$ for positive constants that can vary from line to line.
We write $f \sim g$ to mean $\lim f/g =1$, $f \preceq g$ to mean $f \le Cg$,
$f \succeq g$ to mean $g \preceq f$, and
$f \asymp g$ to mean that $f \preceq g \preceq  f$, where we require that all constants depend only on the dimension $d$ and the spread-out parameter $L$.
Constants depending on additional parameters will be denoted by subscripts, so that, e.g., ``$f(n,\lambda) \asymp_\lambda
g(n,\lambda)$ for every $n \geq 1$ and $\lambda>0$" means that for each $\lambda>0$ there exist positive constants $c_\lambda,C_\lambda$ such that $c_\lambda g(n,\lambda)
\le f(n,\lambda) \leq C_\lambda g(n,\lambda)$ for every $n\geq 1$.
For $a,b\in\R$, we write $a\vee b= \max\{a,b\}$.
We write $\Lambda_r^d=[-r,r]^d \cap \Z^d$ for the box of side length $2r+1$ in $\Z^d$, omitting the $d$ when it is unambiguous to do so.
The \emph{boundary} $\partial \Lambda_r^d$
of $\Lambda_r^d$ consists of the points $x\in\Z^d$ with $\|x\|_\infty=r$.
To avoid dividing by zero, we use the \emph{Japanese bracket} notation $\langle x \rangle:=\|x\|_\infty \vee 1$ for $x\in \Z^d$.
Our notational convention is that objects on the torus have a
label $\T$, so the two-point function on the torus is written as $\tau_p^\T(x)$.
Generally, objects without the torus label are for $\Z^d$.

\smallskip

For $\Z^d$, the restrictions on the dimension $d$ and the range $L$ described in (i) and (ii) above
are so that previous lace expansion results can be applied.
We apply these $\Z^d$ results to the torus under the same restriction.
We apply existing results obtained via the lace expansion for $\Z^d$, and do not
need to revisit or further develop the expansion itself (nor do we use a separate
torus expansion as in \cite{BCHSS05b}).
More precisely, our results hold for any $d>6$ and $L \geq 1$ such that the \emph{two-point function}
$\tau_p(x):=\P_p(0 \leftrightarrow x)$
satisfies
\begin{equation}
\label{eq:twopointassumption}
\tag{T}
\tau_{p_c}(x)  \asymp \langle x \rangle^{-d+2}.
\end{equation}
Here and throughout the paper we write $p_c$ for the critical value for percolation
on $\Z^d$.
The estimate \eqref{eq:twopointassumption} was proven to hold in settings (i) and (ii) above by Hara, van der Hofstad, and Slade \cite{Hara08,HHS03} and Fitzner and van der Hofstad \cite{FH17}.
Our results also rely crucially on those of Kozma and Nachmias \cite{KN11} and Chatterjee and Hanson \cite{CH20}, who worked under the same assumptions.
We will refer to (i) and (ii) collectively as high-dimensional percolation on $\Z^d$,
and to (iii) and (iv) as high-dimensional percolation on the torus.

\subsection{Results for $\Z^d$}

\textbf{The two-point function.} Our first result, and main tool for all our further results, concerns the transition
from exponential to power-law decay for the two-point function.

\begin{theorem}
\label{thm:2pt}
Let $d>6$ and suppose that \eqref{eq:twopointassumption} holds on $\Z^d$.
There exist positive constants $c$ and $C$ such that
\begin{equation}
\label{e:taudecay}
    \tau_p(x) \le \frac{C}{\langle x \rangle^{d-2}} \exp\left[-c(p_c-p)^{1/2}\langle x \rangle \right].
\end{equation}
for every $p\in (0, p_c]$ and $x\in\Z^d$.
\end{theorem}

Theorem~\ref{thm:2pt} is a partial substantiation, via a one-sided bound,
of the generally unproven guiding
principle in the scaling theory for critical phenomena in statistical mechanical models
on $\Z^d$ that two-point functions
near a critical point generically  have decay of the form
\begin{equation}
\label{e:Gscaling}
    \tau_p(x) \approx \frac{1}{\xvee^{d-2+\eta}} g(|x|/\xi(p))
\end{equation}
in some reasonable meaning for ``$\approx$'', when $\xvee$ is
of roughly the same order as
the correlation length $\xi(p)$
and $p$ is close
to its critical value $p_c$.
The universal critical exponent $\eta$ depends on the dimension, the correlation length
$\xi(p) \approx (1-p/p_c)^{-\nu}$
diverges as $p \to p_c$  with
a  dimension-dependent universal critical exponent $\nu$,
and $g$ is a function with rapid decay.
In high dimensions,
$\eta=0$ and $\nu = \frac 12$.
The role of \eqref{e:Gscaling} in
the derivation of scaling relations
between critical exponents, such as Fisher's relation $\gamma=(2-\eta)\nu$,
can be found
in \cite[Section~9.2]{Grim99}.

Let us now summarise how
Theorem~\ref{thm:2pt} compares
to previous results.
For high-dimensional percolation on $\Z^d$,
\eqref{eq:twopointassumption} is known in the more precise asymptotic form
\begin{equation}
\label{e:taupc}
    \tau_{p_c}(x) \sim A_\tau \frac{1}{\langle x\rangle^{d-2}} \qquad (x\to\infty)
\end{equation}
for some positive constant $A_\tau$, with an explicit error estimate  \cite{Hara08,HHS03}.
For all dimensions $d \ge 2$ and for $p<p_c$ there is exponential decay, in the sense that
the \emph{mass} (or \emph{inverse correlation length})
\begin{equation}
\label{eq:massdef}
    m(p) = -\lim_{n\to\infty} \frac 1n \log \tau_p(ne_1)
    = - \sup_{n\geq 1} \frac 1n \log \tau_p(ne_1),
\end{equation}
is strictly positive for $p<p_c$ \cite{Grim99}.
In fact, more is known in general, and it is shown in \cite[Proposition~6.47)]{Grim99}
that there is a constant $c$ such that
\begin{equation}
\label{eq:massbds}
    c p^d \frac{1}{\|x\|_1^{4d(d-1)}}e^{-m(p) \|x\|_1}
    \le
    \tau_p(x) \le e^{-m(p) \|x\|_\infty}
\end{equation}
for all $d \ge 2$, $p \in [0,1]$,  and  $x \neq 0$.
Hara \cite{Hara90} proved that in high-dimensional percolation on $\Z^d$ the mass satisfies the asymptotic formula
\begin{equation}
\label{e:masy}
    m(p) \sim A_m (p_c-p)^{1/2}
    \qquad ( p\to p_c^-)
\end{equation}
for some positive constant $A_m$.
With \eqref{eq:massbds}, this immediately implies that there is a
positive constant $c$ such that
\begin{equation}
\label{e:tauexpbd}
\tau_p(x)
    \leq \exp\left[ -c (p_c-p)^{1/2} \| x \|_\infty \right]
\end{equation}
for every $p<p_c$ and $x\in \Z^d$.  \cref{thm:2pt} improves this bound by a polynomial term which is believed to be sharp.
 No such estimate on the slightly subcritical two-point function had previously been proven for high-dimensional percolation on $\Z^d$.

  For weakly self-avoiding walk in dimension $d>4$, a result analogous to \cref{thm:2pt} was
proved only recently in \cite{Slad20_wsaw}; that proof does not extend to
percolation and our methods are different and do not extend to weakly self-avoiding walk.
On the basis of the asymptotic behaviour for the lattice Green function presented in
\cite{MS21}, we believe that the precise asymptotic behaviour of the subcritical
two-point function for fixed $p<p_c$ and for $d>6$ takes the form
\begin{equation}
\label{e:tauconjecture}
    \tau_p(x) \sim A_{p,\hat x}\; m(p)^{(d-3)/2}\frac{1}{|x|_p^{(d-1)/2}} e^{-m(p)|x|_p}
    \qquad
    (|x|_p \to \infty),
\end{equation}
with $|\cdot|_p$ a $p$-dependent norm on $\R^d$ (\emph{not} the $\ell_p$ norm)
which interpolates monotonically
between the $\ell_1$ and $\ell_2$ norms as $p$ increases over the interval $(0,p_c)$,
and with an amplitude $A_{p,\hat x}$ that approaches a nonzero constant (independent of the
direction $\hat x$) as $p\uparrow p_c$.
The polynomial factors in \eqref{e:tauconjecture}
can be rearranged as $(m(p)|x|_p)^{(d-3)/2} |x|_p^{-(d-2)}$,
so when $|x|_p$ is comparable to the
correlation length $m(p)^{-1}$ the conjectured asymptotic estimate \eqref{e:tauconjecture}
becomes consistent with \eqref{e:taudecay}.

\begin{rk-non} {\rm
It is impossible for the bound \eqref{e:taudecay} to hold for all $x$ if we
replace $c(p_c-p)^{1/2}$ by $m(p)$
in the exponent.  To see this, we recall that
the \emph{Ornstein--Zernike decay} $\tau_p(n,0,\ldots,0) \sim C_p n^{-(d-1)/2} e^{-m(p)n}$
was proved by Campanino, Chayes and Chayes \cite{CCC91} for $p<p_c$
in dimensions $d \ge 2$ (but without the
control  conjectured in
\eqref{e:tauconjecture} for the $p$-dependence of the constant $C_p$ as $p \to p_c$).
From this,  by taking $n$ large we see that \eqref{e:taudecay} can only hold if
the constant $c$ is such that $c(p_c-p)^{1/2}$
is strictly smaller than $m(p)$.}
\end{rk-non}

\medskip

\noindent\textbf{The one-arm probability.} Our second main result for percolation on $\Z^d$ concerns the probability that the cluster of the origin has a large radius in slightly subcritical percolation.
Recall that $\Lambda_r=\Lambda_r^d = [-r,r]^d \cap \Z^d$ is the box of side length $2r+1$ and $\partial \Lambda_r$ is its boundary.

\begin{theorem}
\label{thm:main1arm}
Let $d>6$ and
 suppose that \eqref{eq:twopointassumption} holds on $\Z^d$.
There exist positive constants $c$ and $C$ such that
\begin{equation}
 \frac{c}{r^2} \exp\left(- C (p_c-p)^{1/2} r \right) \leq \P_{p} \bigl( 0 \leftrightarrow \partial \Lambda_r\bigr) \leq   \frac{C}{r^2} \exp\left(- c (p_c-p)^{1/2} r \right)
\end{equation}
for every $p_c/2 \leq p \leq p_c$ and $r \geq 1$.
\end{theorem}

The restriction $p\geq p_c/2$ appearing
in Theorem~\ref{thm:main1arm} is only needed for the lower bound, and $p_c/2$ could be replaced by any constant strictly between $0$ and $p_c$.
The $p=p_c$ case of this theorem was proven by Kozma and Nachmias \cite{KN11} and is used crucially in our proof.
By \eqref{e:tauexpbd} and a union bound,
\begin{equation}
\P_{p} \bigl( 0 \leftrightarrow \partial \Lambda_r\bigr)
 \le 2d(2r+1)^{d-1}
\exp\left(- c (p_c-p)^{1/2} r \right),
\end{equation}
for every $p<p_c$ and $r\geq 2$; the content of the upper bound of \cref{thm:main1arm} is to identify the correct polynomial prefactor. Similar theorems have been established for the distribution of the \emph{volume} and \emph{intrinsic radius} of slightly subcritical percolation clusters in \cite[Section 4]{Hutc20_slightly}; these proofs apply to arbitrary transitive graphs satisfying the triangle condition and are much easier to prove than \cref{thm:main1arm}. See also \cite[Section 5]{Hutc20_slightly} for an overview what is expected to hold for \emph{slightly supercritical} percolation in high dimensions.

\begin{rk-non} {\rm
Shortly after this paper first appeared on the arXiv, we learned of independent work of Chatterjee, Hanson, and Sosoe \cite{CHS22} containing an alternative proof of the subcritical one-arm estimate of \cref{thm:main1arm}. Their work is largely orthogonal to ours, using very different methods to establish \cref{thm:main1arm} and not considering the slightly subcritical two-point function or finite-size scaling on the torus. Their work also establishes several further new results on the chemical distance and the number of spanning clusters in a box, which we do not study here.  Six months later, a third proof of the upper bound of
\cref{thm:main1arm} appeared in \cite{Vann22}.}
\end{rk-non}

\subsection{Results for the torus}
\label{e:sec:torus}

Percolation on the high-dimensional torus has received much attention in
recent years \cite{BCHSS05a,BCHSS05b,HHI07,HHII11,HH17book,HS14}, with considerable related
work on hypercube percolation including \cite{BCHSS04c,HH17book,HN17,HN20}.
That work has concentrated on the torus susceptibility and on questions with a flavour
like those in the literature on the Erd\H{o}s--R\'enyi random graph such as the
cluster size distribution.  It was based on a \emph{torus triangle condition}
and required a lace expansion analysis directly on the torus, with the focus on
an intrinsically defined torus critical point which was later related to the
critical point $p_c$ for $\Z^d$. In the following, we analyse torus percolation
in the vicinity of $p_c$ directly, with our main tool being the near-critical
bound on the $\Z^d$ two-point function provided by Theorem~\ref{thm:2pt}.
At the end of this section, we will discuss the intrinsically defined critical point
and the torus triangle condition.

Our principal focus here is on the torus two-point function
$\tau_p^\T(x) := \P_p^\T(0\conn x)$ (for $x \in \T_r^d$)
and its ``plateau.''
Despite the substantial progress on high-dimensional torus percolation,  a
detailed analysis of the behaviour of the torus two-point function $\tau^\T_p(x)$
within and below the critical window of width $V^{-1/3}$ about $p_c$
has been missing until now.
A sizeable physics literature for related models such as the Ising model
(in dimensions $d>4$) predicts
the existence of a ``plateau'' for the torus two-point function, namely that
within the critical window
$\tau^\T_p(x)$ decays like the $\Z^d$ two-point function $\tau_{p_c}(x)$
for a certain volume-dependent range of $x$ values, but beyond this range $\tau^\T_p(x)$
levels off at an approximately constant value which exceeds $\tau_{p_c}(x)$.
For the Ising model this is discussed, e.g., in \cite{Papa06,ZGFDG18,ZGDG20}
and references therein.  Different behaviour is expected for free boundary conditions,
and has recently been proved for the Ising model in \cite{CJN21}.
The plateau has recently been proven to exist for the simple random walk two-point
function (lattice Green function) on the torus
in all dimensions $d>2$ \cite{Slad20_wsaw,ZGDG20}
and for weakly self-avoiding walk on the torus in dimensions $d>4$ \cite{Slad20_wsaw}.
The plateau is applied in an essential way to analyse the
weakly self-avoiding walk on a torus in dimensions $d>4$ in \cite{MS21}.
The differences between free, bulk and periodic boundary conditions for percolation have
been emphasised in \cite{Aize97}, where the focus is on the maximal cluster size rather than
the two-point function plateau; see also \cite[Section~13.6]{HH17book}.

Before stating our results on the torus two-point function, let us first recall some relevant background on the susceptibility. Let $\chi(p):=\sum_{x\in \Z^d}\tau_p(x)$ be the $\Z^d$ susceptibility, which is known  \cite{AN84,FH17,HS90a} to satisfy the mean-field asymptotics
\begin{equation}
\label{e:chiasy}
    \chi(p) \asymp \frac{1}{1-p/p_c} \qquad ( p \to p_c^-)
\end{equation}
for high-dimensional percolation on $\Z^d$.

By using the estimate \eqref{e:taudecay} on the $\Z^d$ two-point function,
we prove the following theorem.
For notational
convenience, we sometimes evaluate a $\Z^d$ two-point function at a point $x\in\T_r^d$,
with the understanding that in this case we regard
$x$ as a point in $[-r/2,r/2)^d \cap \Z^d$.  This occurs in the statement of
Theorem~\ref{thm:plateau}.

\begin{theorem}[The two-point function plateau]
\label{thm:plateau}
Let $d>6$ and suppose that
\eqref{eq:twopointassumption} holds  on $\Z^d$.
\begin{itemize}
\item \textbf{\emph{Below the scaling window:}} There exist positive constants $c_1$ and $C_1$ depending only on $d$ and $L$
    such that
\begin{equation}
\label{e:plateau_upper}
 \tau^\T_p(x) \leq \tau_p(x) + C_1\frac{\chi(p)}{V} \exp\left[-c_1 m(p)r\right]
\end{equation}
for every $r>2$, every $x \in \T_r^d$,
and every
 $p \in [0, p_c)$.
Moreover, there exist positive constants $A_1$, $A_2$, $c_2$, and $M$ such that
 \begin{equation}
\label{e:plateau_lower}
 \tau^\T_p(x) \geq \tau_p(x) +  c_2 \frac{\chi(p)}{V}
\end{equation}
for every $r>2$, every $x \in \T_r^d$ with $\|x\|_\infty> M$, and every
 $p \in [p_c-A_1 V^{-2/d},p_c-A_2 V^{-1/3}]$.
  \item
\textbf{\emph{Inside the scaling window:}}
For each $0<\delta\leq 1$ and $A >0$, there exist positive constants
$r_0$ and $C_3$  depending only on $d$,
 $L$,
$\delta$, and $A$
such that
\begin{equation}
\label{e:plateaupc_upper}
\tau_p^\T(x) \leq (1+\delta)\tau_{p_c}(x) + \frac{C_3}{V^{2/3}}
\end{equation}
for every
$p \in [0,p_c+ A  V^{-1/3}]$,
$r>r_0$, and $x\in\T_r^d$.
In fact, the upper bound \eqref{e:plateaupc_upper} holds for $p\le p_c$ with $\delta=0$.
In addition, there exists a positive constant $M$ depending only on $d$ and $L$,
and a positive $c_3$ depending on $d,L,$ and $A$ such that
\begin{equation}
\label{e:plateaupc_lower}
    \tau_p^\T(x) \geq (1-\delta)\tau_{p_c}(x)+\frac{c_3}{V^{2/3}}
\end{equation}
for every $r>r_0$, every $x \in \T_r^d$ with $\|x\|_\infty> M$, and every
$p \in [p_c- A V^{-1/3},p_c+ A V^{-1/3}]$.
\end{itemize}
\end{theorem}

With a choice $A\ge A_2$, the above theorem gives upper and lower bounds on
the two-point function throughout the range $p_c-A_1V^{-2/d} \le p \le p_c  +A V^{-1/3}$.
Below the window the bounds involve $\chi(p)/V$, which is infinite at $p_c$, whereas
within the window this constant term is replaced by $V^{-2/3}$.

The upper bound of \eqref{e:plateau_upper} is essentially an immediate consequence of
\eqref{e:taudecay}.  The lower bound also uses \eqref{e:taudecay},
but requires a more involved argument inspired by the method
used for weakly self-avoiding walk in \cite{Slad20_wsaw}.
For the estimates inside the scaling window, we also use the
critical one-arm result
of Kozma and Nachmias \cite{KN11}.

We emphasise that the susceptibility $\chi(p)$ appearing in Theorem~\ref{thm:plateau} is the
susceptibility for $\Z^d$, not for the torus.
The upper bound \eqref{e:plateaupc_upper} at $p=p_c$ and with $\delta=0$ is proved
in \cite[Theorem~1.7]{HS14};  we complement the upper bound with a lower bound of
the same order, and extend these bounds through the entire scaling window.
The following corollary provides a more compact though less precise version of \eqref{e:plateaupc_upper}--\eqref{e:plateaupc_lower}.

\begin{cor}
Let $d>6$ and suppose that
\eqref{eq:twopointassumption} holds  on $\Z^d$.  Then, for all $A>0$  there exists $r_0$ such that if $r>r_0$ then
\begin{equation}
\label{e:plateaupT}
    \tau_{p}^\T(x) \asymp_A \frac{1}{\langle x\rangle^{d-2}} + \frac{1}{V^{2/3}}
\end{equation}
for every $x \in \T_r^d$ and every
$p \in [p_c- A V^{-1/3},p_c+ A V^{-1/3}]$.
\end{cor}

\begin{proof}
The upper bound follows immediately from the upper bound \eqref{e:plateaupc_upper}
together with our assumption \eqref{eq:twopointassumption}.

The lower bound also follows immediately from \eqref{e:plateaupc_lower}
if $\|x\|_\infty > M$, once we assume that $r$ is sufficiently large that
$AV^{-1/3} \le A_1V^{-2/d}$.
If instead $\|x\|_\infty \le M$ then it
is joined to the origin by a path of length at most $Md$. Therefore,
if we choose $r$ large enough that $AV^{-1/3} \le p_c/2$ then we have that
$\tau_{p}^\T(x) \ge (p_c/2)^{Md} \ge (p_c/2)^{Md}\langle  x \rangle^{-(d-2)}$.  If we take $r$ large enough that $V^{-2/3} <M^{d-2}$ then, for $\|x\|_\infty \le M$ we also find that
\begin{equation}
    \tau_{p}^\T(x) \ge \frac 12 \Big( \frac{p_c}{2} \Big)^{Md}
    \Big( \frac{1}{\langle x\rangle^{d-2}} + \frac{1}{V^{2/3}} \Big).
\end{equation}
This completes the proof.
\end{proof}

 Consequently, for
$\langle x\rangle^{d-2} < V^{2/3}$ we have $\tau_{p}^{\T}(x) \asymp \langle x\rangle^{-(d-2)}$, whereas for
$\langle x\rangle^{d-2} > V^{2/3}$ we have $\tau_{p}^{\T}(x) \asymp V^{-2/3}$.  This is the plateau:
the torus
two-point function levels off at an approximately constant value once $x$ is large enough.

There is in fact a hierarchy of plateaux extending \eqref{e:plateaupT}.
Consider $p = p_c - V^{-a}$ with $a \in (\frac 2d, \frac 13]$.
By \eqref{e:chiasy} $\chi(p) \asymp V^a$, and by \eqref{e:masy}
$m(p)r \asymp V^{\frac 1d -\frac a2 } \to 0$
as $r \to \infty$.
For such $p$, the plateau effect occurs as soon as  $\xvee^{d-2} \succeq V^{1-a}$.
When $a=\frac 2d$ the constant terms in \eqref{e:plateau_upper}--\eqref{e:plateau_lower}
are of order $r^{-(d-2)}$, which is the smallest order that $\tau_p(x)$ can achieve
for $x \in \T_r^d$.  If $a< \frac 2d$ then
$m(p)r \to \infty$
and the plateau effect is absent.

The torus susceptibility is defined by $\chi^\T(p)= \sum_{x\in \T_r^d} \tau^\T_p(x)$.
The following corollary of Theorem~\ref{thm:plateau} shows that
$\chi^\T(p) \asymp V^{1/3}$ for $p$ in the scaling window.
It shows that the correct transfer of the bounds on the two-point
in Theorem~\ref{thm:plateau}
from below the window into the window is achieved by replacing the $\Z^d$ susceptibility
by the torus susceptibility.
The corollary
reproduces a result of \cite{BCHSS05a,BCHSS05b} via quite different methods, and without
any torus lace expansion.

\begin{cor}
\label{cor:chiT}
Let $d>6$ and suppose that
\eqref{eq:twopointassumption} holds  on $\Z^d$.  For
any $A >0$ there exists $r_0$ such that
\begin{equation}
    \chi^\T(p) \asymp_A V^{1/3},
\end{equation}
for all $r>r_0$ and all $p\in [p_c-A V^{-1/3},p_c+A V^{-1/3}]$.
\end{cor}

\begin{proof}
This follows immediately by summation of \eqref{e:plateaupT} over $x\in\T_r^d$.
Indeed summation of $\tau_{p_c}(x)\asymp \langle x\rangle^{-(d-2)}$
over the torus yields $r^2 = V^{2/d}$, which is smaller when $d> 6$ than the sum of the
constant term which is $V\cdot V^{-2/3} = V^{1/3}$.
\end{proof}

The cube-root divergence in the volume given by Corollary~\ref{cor:chiT}
for periodic boundary conditions should be
contrasted with the situation for free boundary conditions.  For free boundary
conditions, it is a corollary of the bounds on the finite-volume two-point function in \cite[Theorem~1.2]{CH20} that in our setting of high-dimensional percolation
the susceptibility at $p=p_c$ diverges as $r^2$ rather than $r^{d/3}$.
This is one setting where the controversy in the physics literature detailed, e.g.,
in \cite{ZGFDG18}, is rigorously resolved.

Next, we discuss the torus triangle condition.  In general, on a finite graph
there is no unique definition of a critical value because the critical behaviour
extends over a scaling window of $p$ values.  In the above, we have used the
$\Z^d$ critical value $p_c$ as our reference point for the torus with large volume.
In \cite[(1.7)]{BCHSS05a}, instead,
an \emph{intrinsically} defined critical value $p_\T$
is defined to be the unique solution,
given a fixed $\lambda \in[ V^{-1/3},V^{2/3}]$, of the equation
\begin{equation}
\label{e:pTdef}
    \chi^{\T}(p_{\T}) = \lambda V^{1/3}.
\end{equation}
We are interested in large volume $V$, so given any $\lambda>0$ eventually
we do have $\lambda \in[ V^{-1/3},V^{2/3}]$ and $p_\T$ is then well defined.
Of course, $p_{\T}$ depends on $\lambda$, but only slightly and, as we show
below, $p_{\T}$ is an effective critical point no matter which $\lambda$
is chosen.

The \emph{torus triangle diagram} is
\begin{equation}
    \triangle^{\T}_{p}(x) = \sum_{u,v\in \T_r^d} \tau^{\T}_p(u) \tau^{\T}_p(v-u) \tau^{\T}_p(x-v),
\end{equation}
and it attains its maximum value when $x=0$ by \cite[Lemma 3.3]{AN84}.
The \emph{torus triangle condition} is the statement that $\triangle_{p_{\T}}^\T(x)$
is bounded by a constant independent of $r$ and $x$.
Extensive consequences of the torus triangle condition are derived in
\cite{BCHSS05a,HHI07,HHII11,HH17book}.
These consequences often require the stronger assumption that
\begin{equation}
\label{e:tricond}
    \triangle^{\T}_{p_{\T}}(x) \le \mathbbm{1}(x=0) + a_0
\end{equation}
for $x\in \T_r^d$ and some small context-dependent constant $a_0$. We refer to
the condition \eqref{e:tricond} as the \emph{$a_0$-strong torus triangle condition}.
The (strong) torus triangle condition is proved in \cite{BCHSS05b}
using a finite-graph version of the lace expansion under the assumption that either $d$ is very large or $d>6$ and $L$ is large.  We do not need or use the finite-graph
lace expansion in this paper and give an alternate proof of the torus  triangle condition
based on Theorem~\ref{thm:plateau}.
Our proof of the $a_0$-strong torus triangle condition
relies also on its $\Z^d$ counterpart,  namely that
\begin{equation}
\label{e:tricondZd}
    \triangle_{p_c}(x)  =
    \sum_{u,v\in \Z^d} \tau_{p_c}(u) \tau_{p_c}(v-u) \tau_{p_c}(x-v)
    \le \mathbbm{1}(x=0) + a_0.
\end{equation}
Given any $a_0>0$, the $a_0$-strong triangle condition
\eqref{e:tricondZd} is proved in \cite{HS90a}
if $d\ge d_0$ for the nearest-neighbour model,
and if $d>6$ and $L\ge L_0$ for the spread-out model,
with $d_0$ and $L_0$ sufficiently large depending on $a_0$.

\begin{theorem}
\label{thm:tricon}
Let $d>6$ and suppose that \eqref{eq:twopointassumption} holds on $\Z^d$.
\begin{itemize}
\item
\textbf{$p_\T$ in scaling window:}
There exists $\lambda_0>0$ such that for any $\lambda \in (0,\lambda_0]$
there exist constants $C_1$, $r_0$ (both depending on $\lambda$)
such that if we define $p_\T=p_\T(\lambda)$
via $\chi^\T(p_\T)=\lambda V^{1/3}$ then $|p_\T-p_c| \leq C_1 V^{-1/3}$ for every
$r> r_0$.
\item
\textbf{Torus triangle condition:}
With $\lambda_0,r_0$ as above, there exists $C_2$ such that
 $\triangle_{p_\T(\lambda)}^\T (x)\leq C_2$ for every
$\lambda \in (0,\lambda_0]$, $r>r_0$ and $x\in \T^d_r$.
\item
\textbf{Strong torus triangle condition:}
Fix any $a_0>0$ and assume further that the $\frac 12 a_0$-strong triangle condition
\eqref{e:tricondZd} holds on $\Z^d$.
There exists $\lambda_1>0$ (depending on $a_0$)
such that
$\triangle_{p_\T(\lambda)}^\T
 (x)\leq \mathbbm{1}(x=0)+a_0$ for every
 $\lambda \in (0,\lambda_1]$,
$r> r_0(\lambda,a_0)$
 and $x\in \T^d_r$.
\end{itemize}
\end{theorem}

The \emph{critical window} for the torus $\T_r^d$ was defined in \cite{BCHSS05a}
(see \cite[Theorem~1.3]{BCHSS05a}) to consist
of the values of $p$ lying within distance of order $V^{-1/3}$ from $p_\T$.  It was not
proven at that time that the critical value $p_c$ for $\Z^d$ lies in the window,
and it was not until several years later that this fact was proved \cite{HHII11}.
Theorem~\ref{thm:tricon} gives an alternate proof.
In Theorem~\ref{thm:plateau} we have defined the window as being centred at $p_c$
rather than at $p_\T$.    The fact that $|p_c-p_\T|\preceq V^{-1/3}$ indicates
that either choice of centre can be used.
We emphasise that the analysis of \cite{BCHSS05a,HHII11} relied on performing a separate lace expansion on the torus to establish the torus triangle condition, which our proof bypasses.

\subsection{About the proof}

We now give a brief overview of the structure of the paper and the proofs of our main theorems.

\begin{itemize}
\item In \cref{sec:Zdtwopoint,sec:pioneers} we prove our results concerning near-critical percolation on $\Z^d$, \cref{thm:2pt,thm:main1arm}. Both theorems rely on the notion of \emph{pioneer edges}, which are those through
    which the cluster of the origin enters some halfspace for the first time.

\begin{itemize}
\item In \cref{sec:Zdtwopoint} we formulate an estimate on the expected number of pioneer edges for
    a hyperplane, \cref{thm:totalpioneersexpectation_subcritical}, which we then show to imply \cref{thm:2pt,thm:main1arm} in conjunction with the critical two-point estimate \eqref{eq:twopointassumption} and the critical one-arm results of Kozma and Nachmias \cite{KN11}.
\item In \cref{sec:pioneers} we prove our main estimate on pioneer edges, \cref{thm:totalpioneersexpectation_subcritical}. This proof has two components: First, in \cref{sec:number-pioneers}, we apply diagrammatic methods utilising the halfspace two-point function estimates of Chatterjee and Hanson \cite{CH20} to prove \emph{at criticality} that the expected number of pioneers
    for
    the hyperplane $\{ x : x_1 = n\}$ is bounded by a constant independent of $n$ (\cref{prop:totalpioneersexpectation}). Using this together with the aforementioned one-arm estimates of Kozma and Nachmias \cite{KN11}, we deduce a power-law tail bound with exponent $2/3$
    on the \emph{total} number of pioneer edges at criticality (\cref{lem:criticalpioneers}). Second, in \cref{sec:subcritical_pioneers}, we use the theory of decision trees and the OSSS inequality, which we review in \cref{sec:OSSS}, to obtain a differential inequality applying to the distribution of the total number of pioneers (\cref{lem:differentialinequality}). This differential inequality is of the same form as that obtained for the distribution of the radius by Menshikov \cite{Mens86} (see also \cite{DRT19}) and for the distribution of volume by Hutchcroft \cite{Hutc20}. Using this differential inequality together with our results on the distribution of the number of pioneers at $p_c$, we
 deduce sharp estimates on the distribution of the total number of pioneers at $p_c-\eps$ (\cref{prop:subcritical_pioneers}) and conclude the proof of \cref{thm:totalpioneersexpectation_subcritical}.
\end{itemize}

 \item
In \cref{sec:below_window} we apply our $\Z^d$ results to prove the case of \cref{thm:plateau} in which $p$ lies below the scaling window. This eventually leads us to easily derive in Section~\ref{sec:triangle_cond} the (two versions of the) torus triangle condition presented in \cref{thm:tricon}. In order to apply the $\Z^d$ result of Theorem~\ref{thm:2pt} to the torus we crucially rely on a coupling of percolation on $\Z^d$ and on the torus that was first introduced by Benjamini and Schramm \cite{BS96} and developed extensively in the works of Heydenreich and van der Hofstad \cite{HHI07,HHII11}.
The massive decay of $\tau_p(x)$ for $p<p_c$ then directly gives the upper-bound \eqref{e:plateau_upper} while the lower bound requires fine control of diagrammatic estimates and thus occupies the bulk of this section. Much of this work follows the same general strategy used to analyse
weakly self-avoiding walk on the torus in \cite{MS22} but differs in the details.

 \item
Finally, in \cref{sec:pcbd}, we prove the part of \cref{thm:plateau} in which $p$ lies inside the scaling window. The relevant lower bounds are easy consequences of the `below the scaling window' estimates since $\tau_p^\T(x)$ is monotone in $p$. The upper bounds are proven first at $p_c$ using the coupling between $\Z^d$ and $\T_r^d$ percolation as well as the extrinsic one-arm result from Kozma and Nachmias \cite{KN11} using a variation on the methods of van der Hofstad and Sapozhnikov \cite{HS14}. We then extend the result  for $p \in (p_c, p_c+AV^{-1/3}]$ by the combination of an elementary coupling of Bernoulli percolation at different probabilities and of the input of the intrinsic one-arm exponent controlled in Kozma and Nachmias \cite{KN09}.
\end{itemize}

\section{Near-critical percolation: Proof of Theorems~\ref{thm:2pt}--\ref{thm:main1arm}}
\label{sec:Zdtwopoint}

In this section, we prove Theorems~\ref{thm:2pt}--\ref{thm:main1arm}
subject to Theorem~\ref{thm:totalpioneersexpectation_subcritical}, which concerns
the expected number of \emph{pioneer edges}.

\subsection{Pioneer edges}

For each $n\in \Z$, let $S_n$ be the hyperplane $\{ (y_1,\ldots,y_d) \in \Z^d: y_1 = n \}$ and let $H_n$ be the halfspace $H_n=\{ (y_1,\ldots,y_d)\in \Z^d: y_1 \geq n \}$. We will often write $H=H_0$ to lighten notation.

\begin{defn}
\label{def:pioneer}
Given $x\in \Z^d$, we call an edge $\{y,z\} \in \B$
an $x$-\emph{pioneer} if $x_1 < z_1$, $y_1 < z_1$, $\{y,z\}$ is open, and $x$ is connected to $y$ by an open path contained in the half-space $\{(w_1,\ldots,w_d) : w_1
<
z_1\}$.
That is, $\{y,z\}$ is an $x$-pioneer if $z$ lies to the right of $x$ and
there exists an open path starting at $x$ whose last edge is $\{y,z\}$ with the
path lying strictly to the left of $z$ at every previous time.
For each $x\in \Z^d$ and $n\geq 1$ we define $\mathcal{P}_x(n)$ to be the set of $x$-pioneers $\{y,z\}$
with $y_1 < x_1+n \leq z_1$ and define
 $\mathcal{P}_x=\bigcup_{n\geq 1} \mathcal{P}_x(n)$ to be the set of all $x$-pioneers.
 See Figure~\ref{fig:pioneers}.
\end{defn}

  \begin{figure}
  \centering
\includegraphics{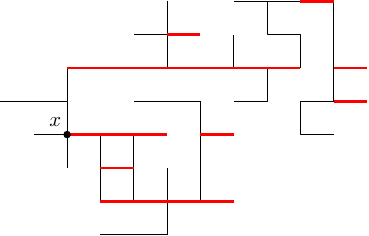}
\caption{Pioneer edges of a cluster in $\Z^2$. Edges that are pioneers with respect to the distinguished vertex $x$ are thick and red, while other edges are thin and black.
\label{fig:pioneers}
 }
 \end{figure}

 Since we are only interested in finite-range models, we have that $\cP_x(n) \cap \cP_x(m)=\emptyset$ when $m \ge
 n+L$ and hence that $\frac{1}{L} \sum_{n\geq 1} |\mathcal{P}_x(n)| \leq |\mathcal{P}_x| \leq \sum_{n\geq 1} |\mathcal{P}_x(n)|$.
For each $p \in [0,1]$ and $n\geq 1$ we define $P_p(n)=\E_p |\cP_0(n)|$, which may be infinite.
We begin by noting that $P_p(n)$ satisfies the following elementary submultiplicativity property.

\begin{lemma}
\label{lem:pioneers_submult}
$P_p(n+m) \leq P_p(n) \max_{0\leq i \leq L-1} P_p(m-i)$ for every $p\in [0,1]$ and $n\geq 1$ and $m\geq L$.
\end{lemma}

It is convenient to simplify the inequality of Lemma~\ref{lem:pioneers_submult},
as follows. Let $e_1=(1,0,\ldots,0)$ be the basis vector in the positive horizontal direction.  First we observe that if $A_m$ denotes the event that
each of the $m$ unit-length
edges of the horizontal path connecting $0$ to $me_1$ are open then
\begin{align}
	\E_p|\cP_0(n+m)|
	&\geq \E_p\Big[|\cP_0(n+m)|\1(A_m) \Big]
    \ge
    \E_p\Big[|\cP_{m e_1}(n)|\1(A_m)\Big]\nnb
	&\geq \E_p|\cP_{m e_1}(n)|p^m
	=  \E_p|\cP_{0}(n)|p^m ,
\end{align}
where we used the Harris--FKG inequality for the third inequality.  A complementary
bound can be obtained by a similar argument, with the result that
\begin{equation}
    p^m P_p(n) \leq P_p(n+m) \leq p^{-m} P_p(n)
\end{equation}
for every $0 < p \leq 1$ and $n,m\geq 1$.  Therefore, by \cref{lem:pioneers_submult},
we obtain the simplified submuliplicative inequality
\begin{equation}
\label{eq:submult_simplified}
P_p(n+m) \leq p^{-L+1}P_p(n)P_p(m)
\end{equation}
for every $0 < p \leq 1$  and $n,m \geq 1$.

\begin{proof}[Proof of \cref{lem:pioneers_submult}]
Suppose that $\{y,z\} \in \cP_0(n+m)$ and that $y_1<z_1$, so that there exists a simple open path $\gamma$ starting at $0$ that has $\{y,z\}$ as its last edge and lies strictly to the left of $z$ at every previous step. Letting $\{a,b\}$ be the first edge crossed by $\gamma$ as it enters $H_n$ for the first time (where $a_1<b_1$), we see that the portion of $\gamma$ up to and including the edge $\{a,b\}$ and the portion of $\gamma$ after this edge are disjoint witnesses for the events that $\{a,b\}\in \cP_0(n)$ and that $\{y,z\}\in \cP_b(n+m-b_1)$. It follows by a union bound and the BK inequality that
\begin{align}
P_p(n+m) &= \sum_{y_1<z_1} \P_p\left( \{y,z\} \in \cP_0(n+m)\right)\\
&\leq \sum_{a_1 < b_1} \P_p\left( \{a,b\} \in \cP_0(n)\right) \sum_{y_1<z_1} \P_p\left( \{y,z\} \in \cP_0(n+m-b_1)\right).
\end{align}
Now, if $a_1<b_1$ and $\{a,b\} \in \cP_0(n)$ then we must have that $n\leq b_1 \leq n+L-1$ and the claim follows by translation-invariance.
\end{proof}

\cref{thm:2pt,thm:main1arm} will both be deduced from the following theorem.

\begin{theorem}
\label{thm:totalpioneersexpectation_subcritical}
Let $d>6$ and suppose
that \eqref{eq:twopointassumption} holds.
There exist positive constants $c$ and $C$ such that
\begin{equation}
 \exp\left[ - C(p_c-p)^{1/2} n \right] \preceq P_p(n)
 \preceq \exp\left[ - c(p_c-p)^{1/2} n \right]
\end{equation}
for every $n\geq 1$ and $p_c/2\leq p \leq p_c$.
\end{theorem}

The lower bound of \cref{thm:totalpioneersexpectation_subcritical}
is an easy
consequence of
\eqref{eq:submult_simplified}, as follows.  First, by
Fekete's lemma,
$-\lim_{n\to\infty} \frac{1}{n} \log P_p(n)$ is well-defined as an element of
$[-\infty,
\infty]$ and satisfies
\begin{equation}
-\lim_{n\to\infty} \frac{1}{n} \log P_p(n) = \sup_{n\geq 1} -\frac{1}{n}\log \left[ p^{-L+1}P_p(n) \right]
\end{equation}
for every $0 < p \leq 1$.
Also,
$\cP_0(n)$ must be nonempty on the event that $0$ is connected to $n e_1$ and hence by Markov's inequality
$\tau_{p_c}(ne_1) \leq P_p(n)$ for every $n\geq 1$. Using this together with \eqref{e:masy} one may verify that the
exponential decay rate of $P_p(n)$ is equal to the mass
$m(p)$ whenever $p<p_c$, and hence that
\begin{equation}
\label{eq:pioneers_Fekete}
p^{L-1} \exp\left[-m(p)n \right] \leq P_p(n) \leq \exp\left[ - m(p) n+o(n) \right]
\end{equation}
as $n\to \infty$ for each fixed $p<p_c$, where
the subexponential correction in the upper bound may depend on the value of $p<p_c$. (Indeed, explicit upper bounds of this form can be deduced from \eqref{e:tauexpbd} by direct summation.) This is of course consistent with \cref{thm:totalpioneersexpectation_subcritical} since $m(p) \asymp |p-p_c|^{1/2}$ as $p\uparrow p_c$ in the high-dimensional setting \cite{Hara08}.
\cref{thm:totalpioneersexpectation_subcritical} eliminates the sub-exponential term from this upper bound for high-dimensional models at the cost of replacing $m(p)$ with $cm(p)$ for some positive constant $c$.  This will be used to obtain the sharp control of the subexponential terms in \cref{thm:2pt,thm:main1arm}.

\subsection{Proof of Theorems~\ref{thm:2pt}--\ref{thm:main1arm}}

\subsubsection{Proof of Theorem~\ref{thm:2pt} and the upper bound of Theorem~\ref{thm:main1arm}}

We now show how \cref{thm:totalpioneersexpectation_subcritical} easily implies \cref{thm:2pt} and the upper bound of \cref{thm:main1arm}.

\begin{proof}[Proof of \cref{thm:2pt} given \cref{thm:totalpioneersexpectation_subcritical}]
Note that it is enough to prove the result for $p \in [p_c/2,p_c]$ so that Theorem \ref{thm:totalpioneersexpectation_subcritical} applies. Indeed, for $p\leq p_c/2$ we can simply use monotonicity to bound $\tau_p(x)\leq \tau_{p_c/2}(x)$ and deduce a bound of the desired form (with a smaller constant in the exponential).

We may assume without loss of generality that the point $x\in \Z^d$ satisfies $x_1=\langle x \rangle \geq 4L \geq 1$. Let $n=\lfloor x_1/2 \rfloor$. Suppose that the origin is connected to $x$ by some simple open path $\gamma$, and let $\{a,b\}$ be the edge that is crossed by $\gamma$ as it enters the halfspace $H_n$ for the first time, with $a_1<b_1$. Then the portion of $\gamma$ up to and including the edge $\{a,b\}$ and the portion of $\gamma$ after this edge are disjoint witnesses for the events $\{a,b\} \in \cP_0(n)$ and $\{b\leftrightarrow x\}$. Thus, we have by a union bound and the BK inequality that
\begin{align}
\P_p(0\leftrightarrow x) &\leq \sum_{a_1 < b_1} \P_p\bigl(\{a,b\} \in \cP_0(n)\bigr) \P_p(b \leftrightarrow x)
\nnb & \preceq \sum_{a_1 < b_1} \P_p\bigl(\{a,b\} \in \cP_0(n)\bigr) \cdot \langle x-b \rangle^{-d+2},
\end{align}
for every $0\leq p\leq p_c$, where we used \eqref{eq:twopointassumption} in the second inequality. Now, if $\{a,b\} \in \cP_0(n)$ then we must have that $n\leq b_1 \leq n+L-1$ and hence that $\langle x-b \rangle \geq x_1 -b_1 \geq x_1/4$, so that there exists a positive constant $c$ such that
\begin{align}
\P_p(0\leftrightarrow x) &\preceq \langle x \rangle^{-d+2} \sum_{a_1 < b_1} \P_p\bigl(\{a,b\} \in \cP_0(n)\bigr) = \langle x \rangle^{-d+2} P_p(n)
\nnb &\preceq \langle x \rangle^{-d+2} \exp\left[-c(p_c-p)^{1/2} \langle x \rangle \right]
\end{align}
by \cref{thm:totalpioneersexpectation_subcritical}.  This completes the proof of \cref{thm:2pt}.
\end{proof}

\begin{proof}[Proof of upper bound of \cref{thm:main1arm} given \cref{thm:totalpioneersexpectation_subcritical}]
Recall that $H_n$ denotes the halfspace $\{y \in \Z^d : y_1 \geq n\}$.
It suffices by symmetry to prove that there exists a positive constant $c$ such that
\begin{equation}
\P_p(0 \leftrightarrow H_{2n}) \preceq \frac{1}{n^2} \exp\left[-c(p_c-p)^{1/2}n\right]
\end{equation}
for every $n\geq 0$ and $p_c/2 \leq p \leq p_c$.
Let $n \geq 2L \geq 1$, and suppose that the origin is connected to the halfspace $H_{2n}$ by some simple open path $\gamma$. Letting  $\{a,b\}$ with $a_1<b_1$ be the edge that is crossed by $\gamma$ as it enters the halfspace $H_n$ for the first time, we observe that the portion of $\gamma$ up to and including the edge $\{a,b\}$ and the portion of $\gamma$ after this edge are disjoint witnesses for the events $\{a,b\} \in \cP_0(n)$ and $\{b\leftrightarrow H_{2n}\}$. Thus, we have by a union bound and the BK inequality that
\begin{align}
\P_p(0\leftrightarrow H_{2n}) \leq \sum_{a_1 < b_1} \P_p\bigl(\{a,b\} \in \cP_0(n)\bigr) \P_p(b \leftrightarrow H_{2n}).
\end{align}
Since $b_1 \leq n+L-1 \leq 3n/2$, we deduce by the main result of \cite{KN11} (i.e., the $p=p_c$ case of \cref{thm:main1arm}) that
\begin{align}
\P_p(0\leftrightarrow H_{2n}) \preceq n^{-2} \sum_{a_1 < b_1} \P_p\bigl(\{a,b\} \in \cP_0(n)\bigr)  = n^{-2} P_p(n) \preceq n^{-2} \exp\left[-c(p_c-p)^{1/2}n\right]
\end{align}
as claimed, where we used \cref{thm:totalpioneersexpectation_subcritical} in the final inequality.
\end{proof}

\subsubsection{Proof of lower bound of Theorem~\ref{thm:main1arm}}

In this section we apply \cref{thm:totalpioneersexpectation_subcritical} to prove the lower bound of \cref{thm:main1arm}.
We give the proof
for the nearest-neighbour model, the general finite-range proof being similar but requiring more involved notation.

We begin with some definitions.
Recall that $S_r$ denotes the hyperplane $\{x\in \Z^d : x_1=r\}$ for each $r\in \Z$. For each $-\infty \leq n \leq m \leq \infty$, let $S_{n,m}$ denote the slab $S_{n,m}:=\bigcup_{i=n}^m S_i$.
For each $r\geq 0$, let $X_r$ be the number of points in the hyperplane $S_r$ that are connected to the origin by an open path lying within the halfspace $S_{-\infty,r}$ and let $Y_r \leq X_r$ be the number of points  in the hyperplane $S_r=\{x\in \Z^d : x_1=r\}$ that are connected to the origin by an open path lying within the slab $S_{-r,r}$.
Since we are working with nearest-neighbour models, every edge in $\cP_0(r+1)$ must be of the form $\{(r,x),(r+1,x)\}$ for some $x\in \Z^{d-1}$, and the edge $\{(r,x),(r+1,x)\}$ belongs to $\cP_0(r+1)$ if and only if it is open and $(r,x)$ is connected to $0$ inside the halfspace lying to the left of $(r,x)$.
From this it follows that
\begin{equation}
\label{e:EXr}
    \E_p X_r = \frac{1}{p}\E_p |\cP_0(r+1)|
\end{equation}
for every $r\geq 0$ and $0< p \leq 1$.

\begin{proof}[Proof of lower bound of \cref{thm:main1arm}]
Let $r \ge 1$.
The lower bound we wish to prove asserts that
\begin{equation}
    \P_{p} \bigl( 0 \leftrightarrow \partial \Lambda_r\bigr)
    \ge
    \frac{c}{r^2} \exp\left(- C (p_c-p)^{1/2} r \right)
    .
\end{equation}
Since $ \{Y_r>0\} \subset \{0 \conn \partial \Lambda_r\} $, it suffices to prove that
the above lower bound holds with instead
$\P_p(Y_r>0)$ on the left-hand side.  We will prove this via the
Cauchy--Schwarz inequality
\begin{equation}
    \P_p(Y_r>0) \ge \frac{\left(\E_p Y_r\right)^2}{\E_p  \left[Y_r^2 \right]}
\end{equation}
together with suitable estimates on the first and second moments of $Y_r$.

It follows from \eqref{e:EXr} and \cref{thm:totalpioneersexpectation_subcritical} that there exist positive constants $c$ and $C$ such that
\begin{equation}
\label{eq:E_pX_r}
\exp\left[-C(p_c-p)^{1/2}r\right] \preceq \E_p X_r \preceq \exp\left[-c(p_c-p)^{1/2}r\right]
\end{equation}
for every $p_c/2 \leq p \leq p_c$ and $r\geq 0$.
We write $\{x \xleftrightarrow{A} y\}$ to mean that $x$ and $y$ are connected by an open path using only vertices of $A$.
Observe that for each $r\geq 1$ and $x\in S_r$ we have the inclusion of sets
\begin{equation}
\{0 \xleftrightarrow{S_{-\infty,r}} x \} \setminus \{0 \xleftrightarrow{S_{-r,r}} x\} \subseteq \bigcup_{y\in S_{-r}} \{0 \xleftrightarrow{S_{-r,r}} y \} \circ \{y \xleftrightarrow{S_{-\infty,r}} x\}.
\end{equation}
Indeed, if the event on the left-hand side of this inclusion holds, $\gamma$ is an open path connecting $0$ and $x$ in $S_{-\infty,r}$, and $y$ is the first point of $S_{-r}$ visited by $\gamma$ then the portions of $\gamma$ before and after visiting $y$ are disjoint witnesses for the events  $\{0 \xleftrightarrow{S_{-r,r}} y \}$ and $\{y \xleftrightarrow{S_{-\infty,r}} x\}$ as claimed. It follows by a union bound, the BK inequality, and translation and reflection symmetry that
\begin{align}
\E_p X_r \leq \E_p Y_r + \E_p Y_r \cdot \E_p X_{2r}
\end{align}
for every $0\leq p \leq 1$ and $r\geq 0$.
Applying the estimate \eqref{eq:E_pX_r} it follows that $\E_p Y_r \asymp \E_p X_r$ for every $p_c/2 \leq p \leq p_c$ and $r\geq 0$ and hence that
\begin{equation}
\label{eq:E_pY_r}
\exp\left[-C(p_c-p)^{1/2}r\right] \preceq \E_p Y_r \preceq \exp\left[-c(p_c-p)^{1/2}r\right]
\end{equation}
for every $p_c/2 \leq p \leq p_c$ and $r\geq 0$.

We turn now to the second moment of the random variable $Y_r$.
 Suppose that $x$ and $y$ are two points in $S_r$ both of which are connected to the origin in $S_{-r,r}$.
There must exist
 a point $z\in S_{-r,r}$ such that the events $\{0 \xleftrightarrow{S_{-r,r}} z\}$, $\{z \xleftrightarrow{S_{-r,r}} x\}$, and $\{z \xleftrightarrow{S_{-r,r}} y\}$ all occur disjointly. It follows by a union bound and the BK inequality that
\begin{align}
\E_p \left[Y_r^2 \right] \leq \sum_{k=-r}^r \sum_{z\in S_k} \P_p(0\xleftrightarrow{S_{-r,r}} z) \sum_{x,y \in S_r} \P_p(z\xleftrightarrow{S_{-r,r}} x) \P_p(z\xleftrightarrow{S_{-r,r}} y)
\end{align}
and hence by
\eqref{eq:E_pX_r} that
\begin{align}
\E_p \left[Y_r^2 \right] \preceq \sum_{k=-r}^r \sum_{z\in S_k} \P_p(0\xleftrightarrow{S_{-r,r}} z) \exp\left[ - 2c(p_c-p)^{1/2} (r-k)\right]
\end{align}
for every $p_c/2\leq p \leq p_c$.
Our next goal is to bound the resulting sum over $z$ for each $-r \leq k \leq r$.
Suppose that $z \in S_k$ for some $-r \leq k \leq r$ and suppose that the origin is connected to $z$ by a simple open path in $S_{-r,r}$. By considering the right-most point that this path visits, we see that there must exist $0\leq a \leq r$ and $w\in S_a$ such that the events $\{0 \xleftrightarrow{S_{-r,a}} w\}$ and $\{w \xleftrightarrow{S_{-r,a}} z \}$ occur disjointly. Thus, applying a union bound and the BK inequality again as above, we obtain that
\begin{align}
\sum_{z\in S_k} \P_p(0\xleftrightarrow{S_{-r,r}} z) &\leq \sum_{a= k \vee 0}^r \sum_{w\in S_a} \sum_{z\in S_k} \P_p(0\xleftrightarrow{S_{-r,a}} w) \P_p(w\xleftrightarrow{S_{-r,a}} z)\nnb
&\leq \sum_{a= k \vee 0}^r \sum_{w\in S_a} \sum_{z\in S_k} \P_p(0\xleftrightarrow{S_{-\infty,a}} w) \P_p(w\xleftrightarrow{S_{-\infty,a}} z)
\end{align}
and a further application of \eqref{eq:E_pX_r} gives that
\begin{align}
\sum_{z\in S_k} \P_p(0\xleftrightarrow{S_{-r,r}} z)
&\preceq \sum_{a= k \vee 0}^r \exp\left[ - c(p_c-p)^{1/2} a - c(p_c-p)^{1/2} (a-k)\right]\nnb
& \preceq r
\exp\left[ - c(p_c-p)^{1/2} |k| \right],
\end{align}
for every $p_c/2 \leq p \leq p_c$, $r\geq 1$,
and $-r \leq k \leq r$. Putting these estimates together we obtain that
\begin{align}
\E_p \left[Y_r^2 \right] &\preceq r
\sum_{k=-r}^r \exp\left[ - 2c(p_c-p)^{1/2} (r-k)-c(p_c-p)^{1/2} |k|\right]\nnb
& \preceq  r^2
 \exp\left[ - c(p_c-p)^{1/2} r \right]
\end{align}
for every $p_c/2\leq p \leq p_c$ and $r\geq 1$.
Putting this together with the lower bound of \eqref{eq:E_pY_r},
we obtain
\begin{equation}
\P_p(Y_r >0) \geq \frac{\left(\E_p Y_r\right)^2}{\E_p  \left[Y_r^2 \right]} \succeq \frac{1}{r^2} \exp\left[ -(2C-c) (p_c-p)^{1/2} r\right]
\end{equation}
for every $p_c/2 \leq p \leq p_c$ and $r\geq 1$.
This completes the proof.
\end{proof}

\section{Expected number of pioneers:
Proof of Theorem~\ref{thm:totalpioneersexpectation_subcritical}}
\label{sec:pioneers}

In this section we complete the proof of Theorems~\ref{thm:2pt}--\ref{thm:main1arm}
by proving \cref{thm:totalpioneersexpectation_subcritical}.

\subsection{The expected number of critical pioneers}
\label{sec:number-pioneers}

\subsubsection{The critical case of \cref{thm:totalpioneersexpectation_subcritical}}

In this section we
prove the $p=p_c$ case of \cref{thm:totalpioneersexpectation_subcritical}.

\begin{prop}
\label{prop:totalpioneersexpectation}
Let $d>6$ and suppose
 that \eqref{eq:twopointassumption} holds.
There exist positive constants $c$ and $C$ such that
$c\leq P_{p_c}(n) \leq C $
for every $n\geq 1$.
\end{prop}

Note that the lower bound $P_{p_c}(n) \geq p_{c}^{L-1}$
holds in every dimension by taking $p\uparrow p_c$ in the estimate \eqref{eq:pioneers_Fekete} above; the main content of the proposition is that a matching upper bound holds in the high-dimensional case.

\medskip

To ease notation, we will prove the upper bound of \cref{prop:totalpioneersexpectation} only for nearest-neighbour percolation. The general proof for finite-range models is very similar but substantially more involved
as one must introduce various additional summations to most calculations. This assumption will be in force for the remainder of
Section~\ref{sec:number-pioneers}.
We write $P(n)=P_{p_c}(n)$ and $\P=\P_{p_c}$ to lighten notation.
Recall that $H$ denotes the half-space $\{(n,x) : n \geq 0, x\in \Z^{d-1}\}$,
and that the edge $\{(n-1,x),(n,x)\}$ belongs to $\cP_0(n)$ if and only if it is open and $(n-1,x)$ is connected to $0$ inside the halfspace lying to the left of $(n-1,x)$.
We again
write $\{x \xleftrightarrow{A} y\}$ to mean that $x$ and $y$ are connected by an open path using only vertices of $A$.
By \eqref{e:EXr},
\begin{equation}
    P(n) = \E|\cP_0(n)|
    =
    \sum_{x\in \Z^{d-1}} p_c \cdot
    \P\Bigl((0,0) \xleftrightarrow{S_{- \infty, n-1}} (n-1,x)\Bigr) .
\end{equation}
By translation and reflection symmetry, this gives (with term-by-term equality)
 \begin{align}
    P(n)
	&=
    \sum_{x\in \Z^{d-1}} p_c\cdot\P\Bigl((0,x) \xleftrightarrow{H} (n-1,0) \Bigr)
    .
 \end{align}
A further translation by $(0,-x)$, followed by replacement of $-x$ by $x$, gives
 \begin{align}
\label{eq:pioneersP(i)def}
    P(n)
	&=
    \sum_{x\in \Z^{d-1}} p_c\cdot\P\Bigl((0,0) \xleftrightarrow{H} (n-1,x)\Bigr).
 \end{align}
This equality makes it convenient for us to consider for each $n\geq 0$ the quantity
\begin{equation}
\label{eq:Pbardef}
\bar P(n):= \frac{1}{p_c} P(n+1)
= \sum_{x\in \Z^{d-1}}\P\Bigl((0,x) \xleftrightarrow{H} (n,0)\Bigr)
=  \sum_{x\in \Z^{d-1}}\P\Bigl((0,0) \xleftrightarrow{H} (n,x)\Bigr)
\end{equation}
instead of $P(n)$ itself. A very similar proof to that of \cref{lem:pioneers_submult} yields that $\bar P$ is submultiplicative in the sense that $\bar P(n+m) \leq \bar P(n)\bar P(m)$ for each $n,m \geq 0$.

In order to upper bound $P(n)$, we will prove a complementary supermultiplicative-type estimate on $\bar P(n)$ via diagrammatic methods. For each $n\geq 0$, define
\begin{equation}
    \bar P^*(n)=\max_{0\leq k \leq n} \bar P(k).
\end{equation}
We deduce \cref{prop:totalpioneersexpectation} from the following two estimates.

\begin{prop}
\label{lem:supermultiplicative}
Let $d>6$ and suppose
 that \eqref{eq:twopointassumption} holds.
There exist positive constants $c>0$ and $\ell \in \N$ such that
 \begin{equation}
    \bar P^*(2n+\ell) \geq c\bar P^*(n)^2
 \end{equation}
 for every $n \ge 0$.
\end{prop}

\begin{lemma}
\label{lem:pioneers_expectation_log}
Let $d>6$ and suppose
that \eqref{eq:twopointassumption} holds.  Then
$\bar P^*(n)
 \preceq \log(n+2)$
for every $n\geq 0$.
\end{lemma}

We now show how \cref{prop:totalpioneersexpectation} follows from
\cref{lem:supermultiplicative,lem:pioneers_expectation_log}.  In brief,
\cref{lem:supermultiplicative} implies that if $\bar P^*$ is unbounded then
it must grow exponentially rapidly.  This contradicts
\cref{lem:pioneers_expectation_log}, so $\bar P^*$ must be bounded, as desired.

\begin{proof}[Proof of \cref{prop:totalpioneersexpectation} given \cref{lem:supermultiplicative,lem:pioneers_expectation_log}]
Let $c>0$ and $\ell \in \N$ be the constants from \cref{lem:supermultiplicative}. If there exists $n\geq \ell$ such that $\bar P^*(n) \geq 2/c$ then we have by induction that
\begin{equation}
\bar P^*(3^kn) \geq \bar P^*(2 \cdot 3^{k-1}n+\ell) \geq \frac{1}{c}2^{2^k}
\end{equation} for every $k\geq 1$. This contradicts \cref{lem:pioneers_expectation_log}, and so we must in fact have that $\bar P^*(n)<2/c$ for every $n\geq \ell$ and hence for every $n\geq 0$ as claimed.
\end{proof}

We now prove \cref{lem:pioneers_expectation_log},
which was used above in the proof of \cref{prop:totalpioneersexpectation}
and which will also be used in the proof of \cref{lem:supermultiplicative}.
The proof is based on the upper bounds
\begin{align}
\P(x \xleftrightarrow{H} y) &\preceq \langle x-y \rangle^{-d+1} &&\text{ for every $x\in \Z^{d}$ with $x_1=0$ and every $y\in H$,}
\label{eq:CH_oneboundary}
\\
\P(x \xleftrightarrow{H} y) &\preceq \langle x-y \rangle^{-d} &&\text{ for every $x,y\in \Z^{d}$ with $x_1=y_1=0$.}
\label{eq:CH_bothboundary}
\end{align}
of Chatterjee and Hanson
\cite[Theorems~7.2 and 1.1(b)]{CH20}, as well as their lower bound \cite[Theorem~1.1(b)]{CH20}
\begin{equation}
\P(x \xleftrightarrow{H} y) \succeq \langle x-y \rangle^{-d+1}
\text{ for every $x\in \Z^{d}$
with $x_1=0$ and every $y\in H$ with $\langle y-x \rangle \leq 2y_1$}.
\label{eq:CH_lower}
\end{equation}
The above bounds are valid for $d>6$ assuming
that \eqref{eq:twopointassumption} holds.

\begin{rk}
\label{rk:CH_extension}
For $d>2$, and given $x,y \in H$ with $y=(y_1,\ldots,y_d)$, we set
$\bar y = (-y_1,y_2,\ldots,y_d)$.
By the method of images (see, e.g., \cite[Proposition~8.1.1]{LL10}),  the half-space
lattice Green function is given by $G_H(x,y)=G(x,y)-G(x,\bar y)$ where the
unrestricted lattice
Green function $G(x,y)$ is asymptotic to a multiple of $|x-y|^{2-d}$.  It is
natural to assume that the critical two-point function has the same behaviour, which suggests
an extended version
\begin{equation}
\P(x \xleftrightarrow{H} y) \preceq \frac{(x_1+1)(y_1+1)}{\langle x-y \rangle^{d}} \qquad \text{ for every $x,y\in H$}
\end{equation}
of the Chatterjee--Hanson bounds
which we believe
to be sharp when
$x_1 \vee y_1 \le K\veee{x-y}$ for some fixed $K>0$.
If this bound were proven, it would be possible to deduce \cref{prop:totalpioneersexpectation} directly by summation.
Although \cite[Theorem~6]{CHS22} proves a strengthened form of the Chatterjee--Hanson half-space two-point function estimate, the
strengthened version is not sharp when both points lie near the boundary, and it remains
an open problem to improve the estimate for the half-space two-point function to an extent where it
could be used to prove our critical pioneers estimate \cref{prop:totalpioneersexpectation} via direct summation.
\end{rk}

\begin{proof}[Proof of \cref{lem:pioneers_expectation_log}]
It suffices to prove that $\bar P(n) \preceq \log (n+2)$ for each
$n \ge 0$.
Let $R=(n+1)^d$.  By \eqref{eq:Pbardef} and \eqref{eq:CH_oneboundary},
\begin{align}
    \bar P(n)
    =
    \sum_{x\in \Z^{d-1}} \P\Bigl((0,x) \xleftrightarrow{H} (n,0)\Bigr)
    &\preceq \sum_{x \in \Lambda_n^{d-1}} (n+1)^{-d+1}
    +
    \sum_{x \in \Lambda_{R}^{d-1} \setminus \Lambda_n^{d-1}} \langle x\rangle^{-d+1}
    \nnb & \qquad
    + \sum_{x \in \Z^{d-1} \setminus \Lambda_{R}^{d-1}} \P\Bigl((n,0) \xleftrightarrow{H} (0,x)\Bigr).
\end{align}
To control the final term, we use the Harris-FKG inequality, \eqref{eq:CH_bothboundary} and \eqref{eq:CH_lower} to obtain that
\begin{align}
\sum_{x \in \Z^{d-1} \setminus \Lambda_{R}^{d-1}} \P\Bigl((0,x) \xleftrightarrow{H} (n,0)\Bigr) &\leq
\sum_{x \in \Z^{d-1} \setminus \Lambda_{R}^{d-1}}
\P\Bigl((0,0) \xleftrightarrow{H} (n,0)\Bigr)^{-1}
\cdot \P\Bigl((0,0) \xleftrightarrow{H} (0,x)\Bigr)\nnb
&\preceq \sum_{x \in \Z^{d-1} \setminus \Lambda_{R}^{d-1}} (n+1)^{d-1}\langle x \rangle^{-d}.
\end{align}
Putting these bounds together and using that $|\{x \in \Z^{d-1}:\langle x \rangle =r\}|=O((r+1)^{d-2})$ for every $r\geq 0$, we deduce that
\begin{align}\bar P(n)=\sum_{x\in \Z^{d-1}} \P\Bigl((0,x) \xleftrightarrow{H} (n,0)\Bigr)
&\preceq 1 +
\sum_{r=n}^R r^{-1}
+ \sum_{r=R}^\infty (n+1)^{d-1}r^{-2}\nnb
&\preceq 1+\log \frac{R+1}{n+1} + \frac{(n+1)^{d-1}}{R} \preceq \log(n+2),
\end{align}
and the proof is complete.
\end{proof}

\subsubsection{Proof of \cref{lem:supermultiplicative}}

In this section, we prove \cref{lem:supermultiplicative}.  As a first step,
we make the following definition.

\begin{defn}
Let $e_1=(1,0,\ldots,0)$ be the unit vector in the horizontal direction. Recall that for each $k\in \Z$, $S_k$ denotes the hyperplane $S_k = \{(k,x) : x\in \Z^{d-1}\}=\{x\in \Z^d :x_1=k\}$ and $H_k$ denotes the halfspace $H_k = \bigcup_{i\geq k} S_i$.
Given $0 \leq k < n$, $x\in S_k$ and $y\in S_{n}$, we say that $x$ is a \emph{good pivotal vertex} for the event $\{0 \xleftrightarrow{H_0} y\}$ if the following hold:
\begin{enumerate}
\item The edge $\{x,x+e_1\}$ is open.
\item $0$ is connected to $x$ in $H_0$ off of the edge $\{x,x+e_1\}$.
\item $x+e_1$ is connected to $y$ in $H_{k+1}$.
\item $0$ is not connected to $y$ in $H_0$ off of the edge $\{x,x+e_1\}$.
\end{enumerate}
\end{defn}

We claim that if $0$ is connected to $y$ in $H_0$ then for each $0\leq k<n$ there is at most one good pivotal vertex $x\in S_k$ for the event  $\{0 \xleftrightarrow{H_0} y\}$. Indeed, if $x$ is a good pivotal vertex then any open path from $0$ to $y$ in $H_0$ must pass through the edge $\{x,x+e_1\}$.
If $x,z\in S_k$ were distinct good pivotal vertices then there would exist simple open paths $\gamma_1$ and $\gamma_2$ connecting $0$ to $y$ in $H_0$ such that $\gamma_1$ visits $S_k$ for the last time at $x$ and $\gamma_2$ visits $S_k$ for the last time at $z$. The concatenation of the portion of $\gamma_1$ up until its visit to $z$ with the portion of $\gamma_2$ after it visits $z$ would therefore be an open simple path connecting $0$ and $y$ in $H_0$ that avoids $x$, contradicting the assumption that $x$ is a good pivotal vertex.

The fact that there is at most one good pivotal vertex implies by \eqref{eq:Pbardef} that
\begin{align}
\bar P(n)=\sum_{y\in S_{n}} \P(0 \xleftrightarrow{H_0} y) &\geq \sum_{x\in S_k} \sum_{y\in S_{n}}  \P(0 \xleftrightarrow{H_0} y,\, \text{$x$ a good pivotal vertex for this event})\nonumber\\
&= \frac{p_c}{1-p_c}\sum_{x\in S_k}\sum_{y\in S_{n}}\P(0 \xleftrightarrow{H_0} x,\, x \nxleftrightarrow{H_{0}} x+e_1, \text{ and } x+e_1 \xleftrightarrow
{H_{k+1}} y)
\end{align}
for every $0 \leq k < n$.
By symmetry,
we have equivalently that
\begin{align}
\label{eq:goodpivotalsformula}
\bar P(n+k)
&\geq \frac{p_c}{1-p_c}\sum_{y\in S_{n}}\sum_{x\in S_{-k}}\P(x \xleftrightarrow{H_{-k}} 0,\, 0 \nxleftrightarrow{H_{-k}} e_1, \text{ and } e_1 \xleftrightarrow{H_{1}} y)
\end{align}
for every $n \geq 1$ and $k\geq 0$.

\medskip

To make use of this inequality, we will first prove the following lemma.
Like many results in high-dimensional percolation, its proof relies on
a bound on the open \emph{triangle diagram}
\begin{equation}
    {\sf T}_p(x) = \sum_{y,z\in \Z^d} \tau_p(y)\tau_p(z-y)\tau_p(x-z)
\end{equation}
at the critical value $p=p_c$.  The triangle diagram was introduced by Aizenman and
Newman in 1984 \cite{AN84} and the finiteness of ${\sf T}_{p_c}(x)$ was proved
in \cite{HS90a} for sufficiently large $d$ for the nearest-neighbour model and for $d>6$
for sufficiently spread-out models, and extended in \cite{FH17} to the nearest-neighbour model
in dimensions $d \ge 11$.  Although historically the
proof of
\eqref{eq:twopointassumption} relied on this finiteness of the triangle diagram, a posteriori
\eqref{eq:twopointassumption} yields (for $d>6$)
\begin{equation}
\label{e:triangle-bd}
    {\sf T}_{p_c}(x) \preceq \sum_{y,z\in \Z^d} \veee{y}^{2-d}\veee{z-y}^{2-d}\veee{x-z}^{2-d}
    \preceq \xvee^{6-d}
\end{equation}
via the elementary convolution estimate \cite[Proposition~1.7]{HHS03}.

Indeed, \cite[Proposition~1.7]{HHS03} states more generally that for each $a,b>0$ with $a+b<d$ there exists a constant $C=C(d,a,b)$ such that
\begin{equation}
\label{e:ab_convolution}
    \sum_{y\in \Z^d} \veee{y}^{a-d}\veee{x-y}^{b-d}
    \leq  C\xvee^{a+b-d}
\end{equation}
for every $x\in \Z^d$, and it follows by applying this estimate twice that
for each $a,b,c>0$ with $a+b+c<d$ there exists a constant $C=C(d,a,b,c)$ such that
\begin{equation}
\label{e:abc_convolution}
    \sum_{y,z\in \Z^d} \veee{y}^{a-d}\veee{z-y}^{b-d} \veee{x-z}^{c-d}
    \leq  C\xvee^{a+b+c-d}
\end{equation}
for every $x\in \Z^d$. The following proof will in fact apply \eqref{e:ab_convolution} with $a,b,c = 2+\eps$ rather than the usual triangle estimate \eqref{e:triangle-bd}.

\begin{lemma}
\label{lem:separated_pivotals}
Let $d>6$ and suppose that \eqref{eq:twopointassumption} holds.
There exists a positive constant $\ell$ such that
\begin{equation}
\label{eq:goodpivotals3}
\sum_{x\in S_{-n}}\sum_{y\in S_{n+\ell}}
\P(x \xleftrightarrow{H_{-n}} 0,\, 0 \nxleftrightarrow{H_{-n}} \ell e_1, \,
\ell e_1 \xleftrightarrow{H_{\ell}} y)
\geq \frac{1}{2} \bar P(n)^2
\end{equation}
for every $n\geq 0$ such that $\bar P (n)= \bar P^*(n)$.
\end{lemma}

\begin{proof}
Fix $n\geq 0$. We follow a variation on the strategy of \cite[Lemma 3.2]{KN09}, illustrated in Figure~\ref{fig:inclusion-exclusion}.
Let $K_{0,n}$ denote the cluster of $0$ in $H_{-n}$ and let $\sC_0$ be the set of finite connected subsets of $\Z^d$ containing $0$.
By conditioning on $K_{0,n}$, we see
that
\begin{equation}
\P(x \xleftrightarrow{H_{-n}} 0,\,
0 \nxleftrightarrow{H_{-n}}
\ell e_1, \,
\ell e_1 \xleftrightarrow{H_{\ell}} y) =
\sum_{\substack{A \in \sC_0 \\ A \ni x }}
\P(K_{0,n} = A, \ell e_1 \xleftrightarrow{H_{\ell}} y)
\mathbbm{1}(\ell e_1 \notin A)
\end{equation}
for each $\ell\geq 1$, $x\in S_{-n}$, and $y\in S_{n+\ell}$. Note moreover that if $A \in \sC_0$ is such that $y \notin A$ then
\begin{equation}
\P(K_{0,n} = A, \ell e_1 \xleftrightarrow{H_{\ell}} y)\mathbbm{1}(\ell e_1 \notin A) = \P(K_{0,n} = A, \ell e_1 \xleftrightarrow{H_{\ell}} y \text{ off $A$})
\end{equation}
where we write ``$\ell e_1 \xleftrightarrow{H_{\ell}} y$ off $A$'' to mean that there is an open path from $\ell e_1$ to $y$ in $H_\ell$ that does not visit any vertex of $A$, including at its endpoints. Since the events $\{K_{0,n} = A\}$ and $\{\ell e_1 \xleftrightarrow{H_{\ell}} y \text{ off $A$}\}$ depend on disjoint sets of edges (namely, those edges with at least one endpoint in $A$ and those edges with neither endpoint in $A$), these two events are independent and we deduce that
\begin{equation}
\P(K_{0,n} = A, \ell e_1 \xleftrightarrow{H_{\ell}} y)\mathbbm{1}(\ell e_1 \notin A) = \P(K_{0,n} = A)\P(\ell e_1 \xleftrightarrow{H_{\ell}} y \text{ off $A$}).
\end{equation}
Next, we observe that
\begin{equation}
\P(\ell e_1 \xleftrightarrow{H_{\ell}} y \text{ off $A$}) = \P(\ell e_1 \xleftrightarrow{H_{\ell}} y) - \P(\ell e_1 \xleftrightarrow{H_{\ell}} y \text{ only via $A$}),
\end{equation}
where we write ``$\ell e_1 \xleftrightarrow{H_{\ell}} y$ only via $A$'' to mean that there is an open path from $\ell e_1$ to $y$ in $H_\ell$ but every such path must visit a vertex of $A$. (This holds in particular if $\ell e_1$ is connected to $y$ in $H_\ell$ and belongs to the set $A$.) It follows that
\begin{align}
\P(x \xleftrightarrow{H_{-n}} 0,\,
0 \nxleftrightarrow{H_{-n}}
\ell e_1, \,
\ell e_1 \xleftrightarrow{H_{\ell}} y)
& =
\sum_{\substack{A \in \sC_0 \\ A \ni x
}} \P(K_{0,n} = A)\P(\ell e_1 \xleftrightarrow{H_{\ell}} y)
\nnb & \qquad
-\sum_{\substack{A \in \sC_0 \\
A \ni x
}} \P(K_{0,n} = A)\P(\ell e_1 \xleftrightarrow{H_{\ell}} y \text{ only via $A$})
\end{align}
and hence that
\begin{align}
\label{eq:onlyon_expansion}
\P(x \xleftrightarrow{H_{-n}} 0,\,
0 \nxleftrightarrow{H_{-n}}
\ell e_1,
\ell e_1 \xleftrightarrow{H_{\ell}} y)
&= \P(
x \xleftrightarrow{H_{-n}} 0  ) \P(\ell e_1 \xleftrightarrow{H_\ell} y)\nnb
& \qquad
-\sum_{\substack{A \in \sC_0 \\
A \ni x
}} \P(K_{0,n} = A)\P(\ell e_1 \xleftrightarrow{H_{\ell}} y \text{ only via $A$})
\end{align}
for every $\ell \geq 1$, $x\in S_{-n}$, and $y\in S_{n+\ell}$.

\begin{figure}[t]
\centering
\includegraphics[width=\textwidth]{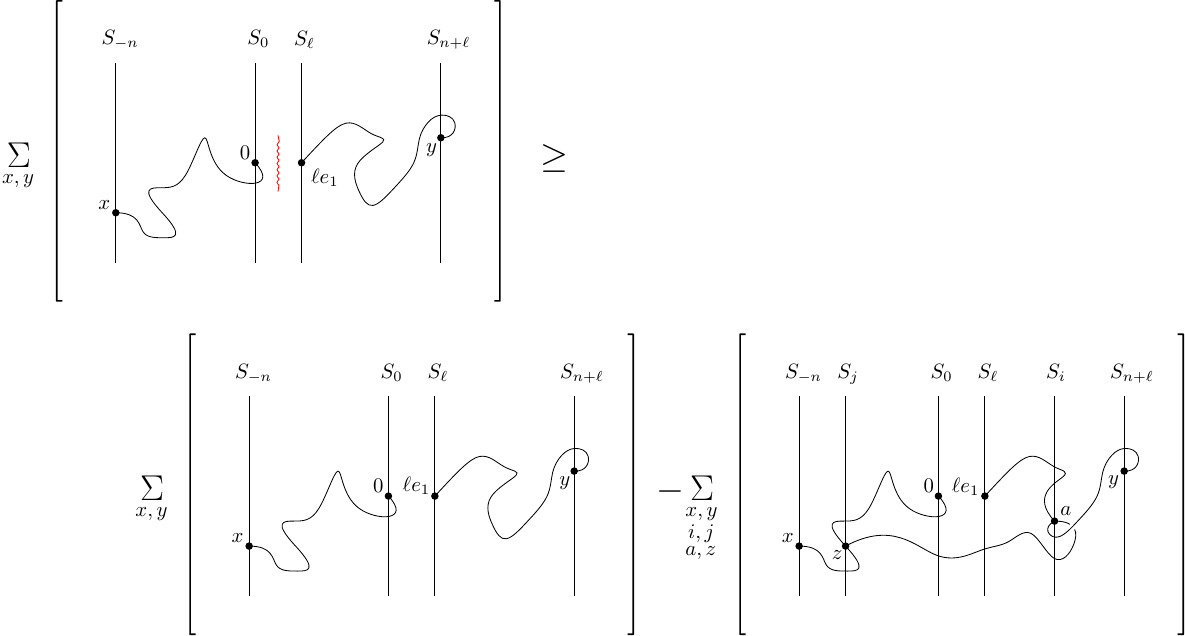}
\caption{Schematic illustration of the diagrammatic estimate used to prove \cref{lem:separated_pivotals}. The squiggly red line indicates that $0$ and $\ell e_1$ are not connected by an open path in the half-space $H_{-n}$. To prove the lemma, it suffices to prove that the second diagrammatic sum on the right hand side is much smaller than the first when the separation parameter $\ell$ is large.}
\label{fig:inclusion-exclusion}
\end{figure}

Our goal is to prove that the sum over $x\in S_{-n}$ and $y\in S_{n+\ell}$
of the left-hand side of \eqref{eq:onlyon_expansion} is bounded below by
$\frac 12 \bar P^*(n)^2$, assuming that $\bar P(n) = \bar P^*(n)$.
For the first term on the right-hand side, it follows from \eqref{eq:Pbardef}
that
\begin{equation}
    \sum_{x\in S_{-n}} \sum_{y\in S_{n+\ell}}
    \P(
    x \xleftrightarrow{H_{-n}} 0
    ) \P(\ell e_1 \xleftrightarrow{H_\ell} y)
    =
    \bar P(n)^2.
\end{equation}
It therefore suffices to prove that we can choose $\ell$ large in order to obtain
\begin{equation}
\label{e:pivgoal}
    \sum_{x\in S_{-n}} \sum_{y\in S_{n+\ell}}
    \sum_{\substack{A \in \sC_0 \\ A \ni x}} \P(K_{0,n} = A)\P(\ell e_1 \xleftrightarrow{H_{\ell}} y \text{ only via $A$})
    \le \frac 12 \bar P(n)^2
\end{equation}
for every $n \ge 0$ such that $\bar P(n) = \bar P^*(n)$.
The remainder of the proof is devoted to establishing \eqref{e:pivgoal}.

As a first step, we observe by the BK inequality that
\begin{align}
\P(\ell e_1 \xleftrightarrow{H_{\ell}} y \text{ only via $A$})
&\leq \sum_{a\in A}\P\left(\{\ell e_1 \xleftrightarrow{H_{\ell}} a\} \circ\{a \xleftrightarrow{H_{\ell}} y\}\right)
\nnb & \leq
\sum_{a\in A}\P\left(\ell e_1 \xleftrightarrow{H_{\ell}} a\right)\P\left(a \xleftrightarrow{H_{\ell}} y\right)
\end{align}
for every $\ell \geq 1$ and $y\in S_{n+\ell}$. Indeed, if the event on the left-hand side occurs then there must exist a simple open path connecting $\ell e_1$ to $y$ in $H_\ell$ that passes through $A$ at some point $a$, and the portions of this path before and after visiting $a$ are disjoint witnesses for the events
$\{\ell e_1 \xleftrightarrow{H_{\ell}} a\}$ and $\{a \xleftrightarrow{H_{\ell}} y\}$. It follows that
\begin{align}
&\sum_{\substack{A \in \sC_0 \\
A \ni x
}} \P(K_{0,n} = A)\P(\ell e_1 \xleftrightarrow{H_{\ell}} y \text{ only via $A$})
\nnb & \qquad
\leq \sum_{\substack{A \in \sC_0 \\
A \ni x
}} \P(K_{0,n} = A) \sum_{a\in A} \P\left(\ell e_1 \xleftrightarrow{H_{\ell}} a\right)\P\left(a \xleftrightarrow{H_{\ell}} y\right)
\nonumber
\\
& \qquad
=\sum_{a\in H_{\ell}}  \P\left(0 \xleftrightarrow{H_{-n}} x , \,
0 \xleftrightarrow{H_{-n}} a \right)\P\left(\ell e_1 \xleftrightarrow{H_{\ell}} a\right)\P\left(a \xleftrightarrow{H_{\ell}} y\right)
\label{e:piv1}
\end{align}
for each $\ell \geq 1$, $x\in S_{-n}$, and $y \in S_{n+\ell}$.
Now, if $0$ is connected to both $x$ and $a$ in $H_{-n}$ there must exist $z\in H_{-n}$ such that the events $\{0 \xleftrightarrow{H_{-n}} z\}$, $\{z \xleftrightarrow{H_{-n}} x\}$, and $\{z \xleftrightarrow{H_{-n}} a\}$ all occur disjointly, so it follows by the BK inequality that
\begin{equation}
\label{e:piv2}
    \P\left(0 \xleftrightarrow{H_{-n}} x , \,
    0 \xleftrightarrow{H_{-n}} a \right)
    \le
    \sum_{z\in H_{-n}}
    \P\left(0 \xleftrightarrow{H_{-n}} z \right)
    \P\left(z \xleftrightarrow{H_{-n}} x \right)
    \P\left(z \xleftrightarrow{H_{-n}} a \right).
\end{equation}

We insert \eqref{e:piv2} into \eqref{e:piv1} and insert the result into \eqref{e:pivgoal}.
The sums over $x$ and $y$ can then be performed explicitly, since these variables each
appear in just one factor.
For the sum over $x$, we use the fact that for $z \in S_{j}$ with $j \ge -n$ we have
\begin{equation}
    \sum_{x\in S_{-n}}
    \P\left(z \xleftrightarrow{H_{-n}} x \right)
    = \bar P (n+j).
\end{equation}
For the sum over $y$, we use that
\begin{multline}
\sum_{y\in S_{r}}\P(0 \xleftrightarrow{H_{-m}} y) \leq \sum_{y\in S_{r}} \sum_{k = -(r \wedge 0)}^m \sum_{w \in S_{-k}} \P(\{0 \xleftrightarrow{H_{-k}} w\} \circ \{w \xleftrightarrow{H_{-k}} y\}) \\\leq \sum_{k= -(r\wedge 0)}^m \bar P(k) \bar P({r}+k)
\end{multline}
for every $m \geq 0$  and $r\geq -m$, which follows by decomposing a simple open path from $0$ to $y$ according to its left-most point and using the BK inequality.
The result is
\begin{align}
\label{e:piv3}
    &\sum_{x\in S_{-n}} \sum_{y\in S_{n+\ell}}
    \sum_{\substack{A \in \sC_0 \\ A \ni x}} \P(K_{0,n} = A)\P(\ell e_1 \xleftrightarrow{H_{\ell}} y \text{ only via $A$})
    \nnb & \quad \le
\sum_{i=\ell}^\infty \sum_{a\in S_i} \sum_{j=-n}^\infty \sum_{z\in S_j} \P\left(0 \leftrightarrow z \right) \bar P(n+j)
\P\left(z \leftrightarrow a \right)\P\left(\ell e_1 \xleftrightarrow{H_{\ell}} a\right)
\nnb &\hspace{8.5cm}
\cdot \sum_{k= 0 \vee (i-n-\ell)}^{i-\ell} \hspace{-1em}\bar P(k) \bar P(n+\ell-i+k)
\nnb & \quad
\leq
 \bar P^*(n)^2 \sum_{i=\ell}^\infty \sum_{a\in S_i} \sum_{j=-n}^\infty \sum_{z\in S_j} \P\left(0 \leftrightarrow z \right) \P\left(z \leftrightarrow a \right)
\P\left(\ell e_1 \xleftrightarrow{H_{\ell}} a\right) (i+1)\bar P(j\vee 0) \bar P^*(i)
,
\end{align}
where in the last step we used $\bar P(k) \le \bar P^*(i)$ and $\bar P(n+\ell-i+k) \le \bar P^*(n)$ for $k\leq i -\ell$,
as well as the submultiplicative property of $\bar P$ to see
that $\bar P(n+j ) \le \bar P^*(n) \bar P(j \vee 0)$.

To estimate the right-hand side of \eqref{e:piv3}, we use
\cref{lem:pioneers_expectation_log} to bound $\bar P(j\vee 0)$ and $\bar P^*(i)$, and
\eqref{eq:twopointassumption} to bound $\P\left(0 \leftrightarrow z \right)$ and
$\P\left(z \leftrightarrow a \right)$.  Also, we use the
half-space
estimate \eqref{eq:CH_oneboundary} to see that
\begin{align}
\P\left(\ell e_1 \xleftrightarrow{H_{\ell}} a\right) (i+1) \bar P^*(i)
&\preceq \langle \ell e_1 - a \rangle^{-d+1} (i+1) \log (i \vee 2).
\end{align}
Here $i \ge \ell$, so $i \le \ell + \veee{\ell e_1-a}$ and therefore
\begin{align}
\label{e:piv4}
\P\left(\ell e_1 \xleftrightarrow{H_{\ell}} a\right) (i+1) \bar P^*(i)
 & \preceq  \langle \ell e_1 - a \rangle^{-d+2} \log(\langle \ell e_1 - a \rangle  \vee 2) + \ell \log ( \ell \vee 2) \langle \ell e_1 -a \rangle^{-d+1}
\nonumber
 \\
 &\preceq \langle \ell e_1 - a \rangle^{-d+ 2+1/4} + \ell^{5/4} \langle \ell e_1 -a \rangle^{-d+1}.
\end{align}
Thus, with the left-hand side of our goal \eqref{e:pivgoal} temporarily written as $T_{n,\ell}$,
using the crude bound $\bar P(j \vee 0) \preceq \log (j\vee 2) \preceq \langle z \rangle^{1/4}$ yields that
\begin{multline}
    T_{n,\ell} \preceq
   \bar P^*(n)^2 \sum_{i=\ell}^\infty \sum_{a\in S_i} \sum_{j=-n}^\infty \sum_{z\in S_j}
   \veee{z}^{-d+2+1/4} \veee{z-a}^{-d+2}
   \langle \ell e_1-a\rangle^{-d+2+1/4}
\\
   +\ell^{5/4}\bar P^*(n)^2\sum_{i=\ell}^\infty \sum_{a\in S_i} \sum_{j=-n}^\infty \sum_{z\in S_j} \langle z \rangle^{-d+2+1/4} \langle z-a\rangle^{-d+2} \langle \ell e_1-a\rangle^{-d+1}
\end{multline}
from which for $d \ge 7$ \eqref{e:abc_convolution} yields
\begin{equation}
    T_{n,\ell} \preceq
    \bar P^*(n)^2 \left(
    \langle \ell e_1 \rangle^{-d+6+1/2} + \ell^{5/4} \langle \ell e_1 \rangle^{-d+5+1/4}\right) \preceq
    \ell^{-d+6+1/2} P^*(n)^2 \leq
    \ell^{-1/2} P^*(n)^2.
\end{equation}
Since this bound holds uniformly over $n\geq 0$ and $\ell \geq 1$, and since the prefactor $\ell^{-1/2}$
tends to zero as $\ell \to \infty$, we deduce that there exists a constant $\ell$ such that
\begin{equation}
    T_{n,\ell} \leq \frac{1}{2}\bar P^*(n)^2.
\end{equation}
This proves \eqref{e:pivgoal} and therefore completes the proof.
\end{proof}

Finally, we deduce \cref{lem:supermultiplicative} from \eqref{eq:goodpivotalsformula} and \cref{lem:separated_pivotals}.
In preparation for this,
inspired by \cite[Section~4]{Hutc20_locality} we define three events and prove a lemma
relating them, as follows.
 Fix any $n \ge 0$, $x\in S_{-n}$ and $y \in S_{n+\ell}$, where
$\ell$ is fixed as in \cref{lem:separated_pivotals}.  We define the event
\begin{equation}
    \sA(x,y)
    = \{x \xleftrightarrow{H_{-n}} 0,\, 0 \nxleftrightarrow{H_{-n}} \ell e_1, \,
    \ell e_1 \xleftrightarrow{H_{\ell}} y \}.
\end{equation}
Then \eqref{eq:goodpivotals3} can be rewritten more compactly as
\begin{equation}
\label{eq:goodpivotals3A}
\sum_{x\in S_{-n}}\sum_{y\in S_{n+\ell}}
\P(\sA(x,y))
\geq \frac{1}{2} \bar P(n)^2
\end{equation}
for every $n\geq 0$ such that $\bar P (n)= \bar P^*(n)$.

 Let $\eta$ be the left-directed horizontal geodesic connecting $\ell e_1$ to $0$, and for each $1\leq i \leq \ell$ let $\eta_i$ be the $i$th
 edge crossed by $\eta$.
  Given a Bernoulli bond percolation configuration $\omega$
  on $\Z^d$,
 let $\omega^i$ be the configuration obtained from $\omega$ by setting
\begin{equation}
    \omega^i(e) =\begin{cases} 1 & e \in \{ \eta_{j}
                                                            : 1\leq j  \leq i\}\\
\omega(e) & e \notin \{\eta_j : 1\leq j  \leq i\}.
\end{cases}
\end{equation}
In particular, $\omega^0=\omega$.
For each $1\leq i \leq \ell$, let $\sB_{i}(x,y)$ be the event that that $0$ and $\ell e_1$ are connected in $H_{-n}$ in $\omega^i$ but not in $\omega^{i-1}$,
 $0$ is connected to $x$ in $H_{-n}$ in $\omega^{i-1}$, and $\ell e_1$ is connected to $y$ in $H_\ell$ in  $\omega^{i-1}$.
Finally, for each $1\leq i \leq \ell$
let
\begin{equation}
\sC_{i}(x,y)= \left\{x \xleftrightarrow{H_{-n}} (\ell-i) e_1,\, (\ell-i)e_1 \nxleftrightarrow{H_{-n}} (\ell+1-i)e_1,
\,
(\ell+1-i)e_1 \xleftrightarrow{H_{\ell+1-i}} y\right\}.
\end{equation}
The events $\sB_i(x,y)$ and $\sC_i(x,y)$ are depicted in Figure~\ref{fig:BC}.

\begin{figure}[t]
\includegraphics[width=0.47\textwidth]{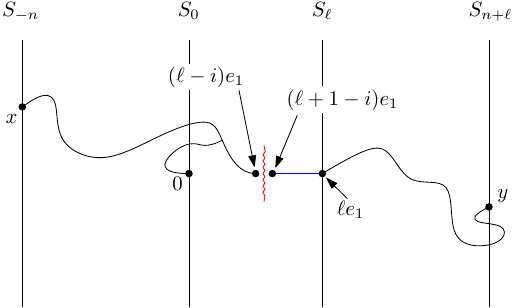}\hfill \includegraphics[width=0.47\textwidth]{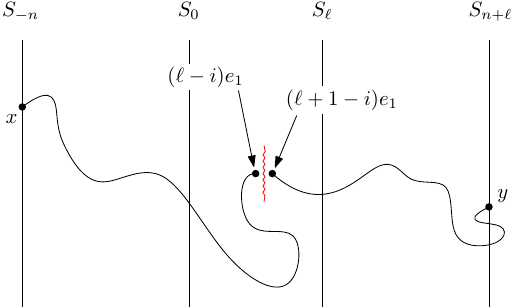}
\caption{Schematic illustrations of the events $\sB_i(x,y)$ (left) and $\sC_i(x,y)$ (right). The blue edges represent those that are forced to be open in $\omega^{i-1}$. The squiggly red line indicates that $(\ell-i)e_1$ and $(\ell+1-i)e_1$ lie in distinct clusters in the half-space $H_{-n}$.}
\label{fig:BC}
\end{figure}

\begin{lemma}
\label{lem:ABC}
With the above setup, and with $p=p_c$,
\begin{equation}
\P\left( \sA(x,y) \right) \leq \sum_{i=1}^\ell p_c^{-i+1} \P\bigl(\sC_{i}(x,y)) .
\end{equation}
\end{lemma}

\begin{proof}
Given a configuration $\omega$, let $i$ be minimal such that $0$
and $\ell e_1$ are connected in $\omega^i$. When the event $\sA(x,y)$ holds,
$i$ cannot be zero, and hence must be between $1$ and $\ell$. Since the
clusters of $0$ and $\ell e_1$ are both larger in $\omega^{i-1}$ than they are in
$\omega$, we must have that $0$ is connected to $x$ in $H_{-n}$ in $\omega^{i-1}$, and
$\ell e_1$ is connected to $y$ in $H_{\ell}$ in $\omega^{i-1}$, which means that
$\sB_i(x,y)$ holds.
It follows that
\begin{equation}
\label{eq:ABsetinclusion}
\sA(x,y) \subseteq \bigcup_{i=1}^{\ell}
\sB_{i}(x,y).
\end{equation}
Since we also have the inclusion of events $\sC_{i}(x,y) \supseteq \sB_{i}(x,y) \cap \{\omega(\eta_j)=1$ for every  $1\leq j \leq i-1\}$,
 and since the two events on the right of this inclusion are independent, we have that
\begin{equation}
\P\bigl(\sC_{i}(x,y)) \ge
p_c^{i-1} \P\bigl( \sB_{i}(x,y)) .
\end{equation}
With \eqref{eq:ABsetinclusion}, this completes the proof.
\end{proof}

\begin{proof}[Proof of \cref{lem:supermultiplicative}]
It suffices to prove that there exist positive constants $c>0$ and $\ell \in \N$ such
that $\bar P^*(2n+\ell) \geq c\bar P^*(n)^2$  for every $n\geq 0$.
Let $\ell$ be as in \cref{lem:separated_pivotals},
and suppose that $n\geq 0$
has $\bar P(n)= \bar P^*(n)$.
Constants in this proof are permitted to depend on $\ell$.
In view of \eqref{eq:goodpivotals3A},
the desired inequality will follow, for such $n$, if we show that
\begin{align}
\label{e:sm1}
\sum_{x\in S_{-n}}\sum_{y\in S_{n+\ell}}
\P\left( \sA(x,y) \right)
 \preceq (\ell+1) \bar P^*(2n +\ell-1).
\end{align}
However this is in fact sufficient for general $n \geq 0$,
since we may take $0
\leq n'\leq n$ such that $\bar P(n')= \bar P^*(n)$ to then deduce that
\begin{equation}
\bar P^*(2n+\ell) \geq \bar P^*(2n'+\ell) \succeq \bar P(n')^2 = \bar P^*(n)^2
\end{equation}
for every $n\geq 0$ as claimed.

It remains to prove \eqref{e:sm1}.  Since both sides of the inequality are positive and the right hand side is finite by \cref{lem:pioneers_expectation_log}, it suffices to consider the case $n\geq 1$. By Lemma~\ref{lem:ABC},
\begin{equation}
\P\left(\sA(x,y)\right)  \preceq \sum_{i=1}^\ell \P\bigl(\sC_{i}(x,y)).
\end{equation}
By translation invariance applied to the event $\sC_i(x,y)$, this gives
\begin{align}
\label{eq:aftersurgery}
&\sum_{x\in S_{-n}}\sum_{y\in S_{n+\ell}}\P\left(\sA(x,y)  \right)
\nnb & \quad
\preceq \sum_{i=1}^\ell \sum_{x \in S_{-n-\ell+i}} \sum_{y\in S_{n+i-1}}\P\left(x \xleftrightarrow{H_{-n-\ell+i}} 0 ,\, 0 \nxleftrightarrow{H_{-n-\ell+i}} e_1,
\,
e_1 \xleftrightarrow{H_{1}} y \right)
 .
\end{align}
Using the assumption that $n\geq 1$, we have by \eqref{eq:goodpivotalsformula} that the right-hand side of \eqref{eq:aftersurgery} is bounded above by
\begin{align}
 &
 \sum_{i=1}^\ell \frac{1-p_c}{p_c} \bar P(2n+ \ell- i)
 \preceq (\ell+1) \bar P^*(2n+\ell-1)
\end{align}
This proves \eqref{e:sm1} and therefore completes the proof.
\end{proof}

\subsection{Proof of \cref{thm:totalpioneersexpectation_subcritical}}
\label{sec:subcritical_pioneers}

\subsubsection{Randomised algorithms and the OSSS inequality}
\label{sec:OSSS}

Our deduction of \cref{thm:totalpioneersexpectation_subcritical} from \cref{prop:totalpioneersexpectation} relies crucially on the OSSS inequality of O'Donnell, Saks, Schramm, and Servedio \cite{OSSS05}, which we now briefly review. This inequality has recently been recognised as a powerful and flexible tool in the study of critical and near-critical percolation models following the breakthrough work of Duminil-Copin, Raoufi, and Tassion \cite{DRT19}.
We build in particular on the techniques developed to apply this inequality to prove inequalities between critical exponents in
\cite{Hutc20}.

Let $E$ be a countable set.
  Informally, a \emph{decision tree} is a deterministic procedure for querying the values of $\omega \in \{0,1\}^E$ that starts by querying the value of some fixed element of $E$ and chooses which element of $E$ to query at each subsequent step as a function of the values it has already observed. Formally, a \emph{decision tree} is a function
  $T:\{0,1\}^E \to E^\N$ from subsets of $E$ to infinite $E$-valued sequences
  $T=(T_1,T_2,\ldots)$ such that
$T_1(\omega) = e_1$ for some $e_1 \in E$ not depending on $\omega$, and such that for each $n \geq 2$ there exists a function $S_n : (E\times \{0,1\})^{n-1} \to E$ such that
\begin{equation}
    T_n(\omega) = S_n\left[\left(T_i,\omega(T_i)\right)_{i=1}^{n-1}\right],
\end{equation}
where we think of $T_n(\omega)$ as the element of $E$ that is queried at time $n$ when given $\omega$ as an input to the procedure.

Let $\mu$ be a probability measure on $\{0,1\}^E$ and let $\omega$ be a random variable with law $\mu$.
For each decision tree $T$ and $n\geq 1$ we define $\cF_n(T)$ to be the $\sigma$-algebra generated by the random  variables $\{T_i(\omega) : 1 \leq i \leq n\}$ and define $\cF(T) =\bigcup_{n\geq 1} \cF_n(T)$. We say that $T$ \emph{computes} a measurable function $f:\{0,1\}^E\to \R$ if $f(\omega)$ is measurable with respect to the $\mu$-completion of the $\sigma$-algebra $\cF(T)$. This is equivalent by L\'evy's 0-1 law to the statement that
\begin{equation}
\mu\left[f(\omega) \mid \cF_n(T)\right] \xrightarrow[n\to \infty]{} f(\omega) \qquad \mu\text{-a.s.}
\end{equation}

To allow for exploration algorithms that are naturally described as parallel rather than serial algorithms, it is convenient to introduce the slightly more general notion of \emph{decision forests}.
A \emph{decision forest} is defined to be a collection of decision trees $F=\{T^i : i \in I\}$ indexed by a countable set $I$. Given a decision forest $F=\{T^i : i \in I\}$ and a probability measure $\mu$ on $\{0,1\}^E$ we let $\cF(F)$ be the smallest $\sigma$-algebra containing all of the $\sigma$-algebras $\cF(T^i)$ and say that a measurable function $f : \{0,1\}^E\to \R$ is \emph{computed by} $F$ if it is measurable with respect to the $\mu$-completion of the $\sigma$-algebra $\cF(F)$.

Let $E$ be a countable set, let $\mu$ be a probability measure on $E$, and let $F=\{T^i:i\in I\}$ be a decision forest on $E$.
For each $e\in E$, we define the \emph{revealment probability}
\begin{equation}
    \delta_e(F,\mu)
 = \mu\left(\text{there exists $i\in I$ and $n \geq 1$ such that } T_{n}^i(\omega)=e\right),
\end{equation}
 so that $\delta_e(F,\mu)$ is the probability that the status of $e$ is ever queried when implementing the decision forest $F$ on a sample from the measure $\mu$.
Finally, we define for each probability measure $\mu$ on $\{0,1\}^E$ and each pair of measurable functions $f,g:\{0,1\}^E\to \R$ the quantity
\begin{equation}
\CoVr_\mu[f,g] = (\mu \otimes \mu) \left[|f(\omega_1)-g(\omega_2)|\right]- \mu  \left[|f(\omega_1)-g(\omega_1)|\right]
\end{equation}
where $\omega_1,\omega_2$ are drawn independently from the measure $\mu$.
Thus,  if $f$ and $g$ are $\{0,1\}$-valued then
\begin{equation}
\label{eq:CoVr_Cov}
\CoVr_\mu[f,g] = 2 \Cov_\mu[f,g] = 2\mu\bigl(f(\omega)=g(\omega)=1\bigr)-2\mu\bigl(f(\omega)=1\bigr)\mu\bigl(g(\omega)=1\bigr).
\end{equation}

We are now ready to state the version of the OSSS inequality that we will use, which is a special case of \cite[Corollary~2.4]{Hutc20}.

\begin{theorem}[OSSS for decision forests]
\label{cor:OSSS_forest}
Let $E$ be a finite or countably infinite set and let $\mu$ be a product measure on $\{0,1\}^E$. Then for every pair of measurable, $\mu$-integrable functions $f,g : \{0,1\}^E \to \R$ and every decision forest $F$ computing $g$ we have that
\begin{equation}
\label{eq:OSSS}
 \sum_{e\in E} \delta_e(F,\mu) \Cov_\mu\left[f,\omega(e)\right] \geq \frac{1}{2}\left|\CoVr_\mu\left[f,g\right]\right|.
\end{equation}
\end{theorem}

See \cite{DRT19}
for an extension
of the OSSS inequality to \emph{monotonic measures} such as the law of the Fortuin-Kastelyn random-cluster model.

\subsubsection{Differential inequalities for Dini derivatives}

In order to discuss how the OSSS inequality leads to differential inequalities in the infinite-volume setting (without any need for finite-volume approximation and limit),
it is convenient to introduce the notion of \emph{Dini derivatives}; see, e.g., \cite{KK96} for further background. The \emph{lower-right Dini derivative} of a function $f:[a,b] \to \R$ is defined to be
\begin{equation}
\lrDini{x}f(x) =  \liminf_{\eps \downarrow  0} \frac{f(x+\eps)-f(x)}{\eps}
\end{equation}
for each $x \in [a,b)$.
In our setting, it is a classical and elementary fact \cite[Theorem 2.34]{Grim99} that if $A$ is an event depending on at most finitely many edges then $\P_p(A)$ is a polynomial in $p$ with derivative
\begin{equation}
\frac{d}{dp}\P_p(A) = \frac{1}{p(1-p)} \sum_{e\in E} \Cov\left[\omega(e),\1(A)\right].
\end{equation}
If $A$ is an \emph{increasing} event depending possibly on infinitely many edges, we still have the lower-right Dini derivative bound
\begin{equation}
\label{e:Dinilb}
\lrDini{p}\P_p(A) \geq \frac{1}{p(1-p)} \sum_{e\in E} \Cov\left[\omega(e),\1(A)\right].
\end{equation}
A detailed proof of this is given in \cite[Proposition 2.1]{Hutc20}.
Thus, the OSSS inequality allows us to prove lower bounds on derivatives of increasing events by exhibiting decision forests that compute these events and have small maximum revealment.

Lower bounds on the lower-right Dini derivatives of monotone functions can often be used in much the same way as bounds on the classical derivative of a differentiable function.
For example, the usual
logarithmic derivative formula
\begin{equation}
\label{e:Dinilog}
\lrDini{x} \log f(x) = \frac{1}{f(x)}\lrDini{x}f(x)
\end{equation}
remains valid.  Also,
if $f:[a,b]\to \R$ is increasing then
\begin{equation}
\label{e:DiniFTC}
f(b) - f(a) \geq \int_{a}^b \lrDini{x}f(x)  d x.
\end{equation}
Since every measurable function has measurable Dini derivatives
\cite[Theorem~3.6.5]{KK96}, the above integral is well-defined.

\subsubsection{Slightly subcritical pioneers: Proof of \cref{thm:totalpioneersexpectation_subcritical}}

\medskip

Our goal in this section is to study the distribution of the total number of $0$-pioneers $|\cP_0|$ in critical and slightly subcritical percolation.
The
main result is the following proposition, which strengthens \cref{thm:totalpioneersexpectation_subcritical}.

\begin{prop}
\label{prop:subcritical_pioneers}
Let $d>6$ and suppose that \eqref{eq:twopointassumption} holds.
There exists a positive constant $c$ such that
\begin{equation}
\E_{p_c-\eps}
|\cP_0| \preceq  \eps^{-1/2} \qquad \text{ and } \qquad
\P_{p_c-\eps}\bigl(
|\cP_0| \geq k\bigr) \preceq  k^{-2/3} \exp \left[ -c \eps^{3/2} k \right]
\end{equation}
for every $0 \leq \eps \leq p_c/2$ and $k\geq 1$.
\end{prop}

The exponential tail bound on $|\cP_0|$ is not needed for the proofs of the main theorems but is included since
it
may be of independent interest.
Before proving this proposition, we show
how it implies \cref{thm:totalpioneersexpectation_subcritical}.

\begin{proof}[Proof of \cref{thm:totalpioneersexpectation_subcritical} given \cref{prop:subcritical_pioneers}]
Recall that $P_p(n) = \E_p |\cP_0(n)|$.
We have already observed that the lower bound of \cref{thm:totalpioneersexpectation_subcritical}
holds,
and we have already proved the desired upper bound when $p=p_c$
in \cref{prop:totalpioneersexpectation}.  It therefore
suffices to prove
that there exists a positive constant $c$ such that
\begin{equation}
 P_p(n) \preceq \exp\left[ - c(p_c-p)^{1/2} n \right]
\end{equation}
for every $n\geq 1$ and $p_c/2\leq p < p_c$.
Fix $p$ in this interval.
As discussed
below Definition~\ref{def:pioneer},
$\cP_0(n) \cap \cP_0(m) = \emptyset$ when $|n-m| \geq L$ and hence
by Proposition~\ref{prop:subcritical_pioneers}
\begin{equation}
\frac{1}{N}\sum_{n = 1}^N  P_p(n)
\leq \frac{L}{N} \E_p |\cP_0| \preceq \frac{1}{(p_c-p)^{1/2} N}
\end{equation}
for every
$N\geq 1$. It follows that there exists a constant $C$ such that
there exists $n_p$ with $1\leq n_{p}
\leq C(p_c-p)^{-1/2}$ such that $P_p(n_p) \leq p^{L-1} /2$.
 It follows inductively by the submultiplicativity estimate \eqref{eq:submult_simplified} that $P_p(kn_p) \leq p^{L-1} 2^{-k}$ for every $k\geq 1$. Since we also have that $P_p(n) \leq P_{p_c}(n) \preceq 1$,
 it follows by another application of \eqref{eq:submult_simplified} that
$P_p(kn_p+r) \preceq 2^{-k}$ for every
$k\geq 1$ and $0 \leq r < n_p$.
If now we write arbitrary $n$ as $n=\lfloor \frac{n}{n_p} \rfloor  n_p +r $,
then we see that the above gives $P_p(n) \preceq  2^{-\lfloor n/n_p \rfloor}$
and the desired exponential estimate follows from the upper bound $n_p \preceq (p_c-p)^{-1/2}$.
\end{proof}

To begin the proof of \cref{prop:subcritical_pioneers} we first note that
 \cref{prop:totalpioneersexpectation}, together with the
 result of Kozma and Nachmias \cite{KN11} that Theorem~\ref{thm:main1arm} holds
 for $p=p_c$, yield
 the following important corollary describing the distribution of the total number of pioneers \emph{at criticality}.

\begin{lemma}
\label{lem:criticalpioneers}
Let $d>6$ and suppose that \eqref{eq:twopointassumption} holds.
Then $\P_{p_c}(|\mathcal{P}_x| \geq k) \preceq k^{-2/3}$ for every $x\in \Z^d$ and $k\geq 1$.
\end{lemma}

\begin{proof}
It suffices to consider the case $x=0$.
Let $n,k \geq 1$, where $n$ is a parameter we will optimise over shortly.
By Markov's inequality,
\begin{align}\P_{p_c}(|\mathcal{P}_0| \geq k) &\leq \P_{p_c}(\text{$|\mathcal{P}_0| \geq k$ and $0$ is not connected to $H_n$} ) + \P_{p_c}(0 \text{ is connected to } H_n )
\\
& \leq \frac{1}{k} \E_{p_c} \left[\sum_{i=0}^n |\cP_0(i)| \right] + \P_{p_c}(0 \text{ is connected to } H_n )\preceq \frac{n}{k} +  \frac{1}{n^2},\end{align}
 for every $k,n \geq 1$, where the  first bound follows from \cref{prop:totalpioneersexpectation} and the second follows from
 the aforementioned result of Kozma and Nachmias. The claim follows by taking $n=\lceil k^{1/3} \rceil$.
 \end{proof}

We now apply the OSSS inequality to deduce \cref{prop:subcritical_pioneers} from \cref{lem:criticalpioneers}. We follow closely the proof of \cite[Theorem 1.1]{Hutc20}.
For each $p\in [0,1]$ and $h\geq 0$, we write $\P_{p,h}$ for the law of the pair $(\omega,\cG)$ where $\omega$ is distributed as Bernoulli-$p$ bond percolation and $\cG$ is a \emph{ghost field} independent of $\omega$, that is, a random subset of the edge set $\B$ in which each edge is included independently at random with inclusion probability $1-e^{-h}$. We call an edge \emph{green} if it belongs to $\cG$. (While it is more standard to consider ghost fields to be random sets of \emph{vertices}, it is more convenient for our purposes to take them to be random sets of edges.)  As a first and key step, we use
the OSSS inequality to prove a differential inequality.

\begin{lemma}
\label{lem:differentialinequality}
The differential inequality
\begin{equation}
\label{eq:pioneersdiffineq}
\lrDini{p} \log \P_p(|\cP_0| \geq k) \geq  \frac{1}{2p(1-p)}
\left[\frac{k(1-e^{-1})}{\sum_{i=0}^{k}
\P_{p}
(|\cP_0| \geq i)}-1\right]
\end{equation}
holds for every $k\geq 1$ and $0<p<1$.
\end{lemma}

\begin{proof}
By \eqref{e:Dinilb} and \eqref{e:Dinilog},
\begin{equation}
    \lrDini{p}\log \P_p(|\cP_0| \geq k)
    \geq
    \frac{1}{\P_p(|\cP_0| \geq k) }
    \frac{1}{p(1-p)} \sum_{e\in \B} \Cov\left[\omega(e), \1(|\cP_0| \geq k) \right]
\end{equation}
where $\B$ is the set of edges.
It therefore suffices to prove that
\begin{equation}
\label{e:di-goal}
    \sum_{e\in \B} \Cov\left[\omega(e), \1(|\cP_0| \geq k) \right]
    \ge
    \frac{1}{2}
    \left[\frac{k(1-e^{-1})}{\sum_{i=0}^{k} \P_{p}(|\cP_0| \geq i)}-1\right]
    \P_p(|\cP_0| \geq k)
    .
\end{equation}
For this we will use Theorem~\ref{cor:OSSS_forest}.

For the setup for Theorem~\ref{cor:OSSS_forest}, we let
$(\omega,\cG)$ have law $\P_{p,h}$, where we think of $\P_{p,h}$ as a product measure on $\{0,1\}^{\B \times \{\mathrm{perc},\mathrm{ghost}\}}$. Consider the two Boolean functions
\begin{equation}f(\omega,\cG)=f(\omega) = \mathbbm{1}(|\cP_0| \geq k)
\qquad\text{and}\qquad
g(\omega,\cG) = \mathbbm{1}(\cP_0 \cap \cG \neq \emptyset).
\end{equation}
We say that an edge is \emph{horizontal} if its endpoints have distinct first coordinates.
We can determine the value of $g$ by first revealing the value of the ghost field at each horizontal edge and then exploring the cluster of each green horizontal edge in the halfspace lying strictly to the left of its rightmost endpoint. This exploration process can be encoded as a decision forest $F=\{T^e : e\in \B\}$ in which the decision tree $T^e$ first queries the status of the ghost field at the edge $e$, halting if it discovers that $\cG(e)=0$. If the decision tree discovers that $\cG(e)=1$, it next checks whether $e$ is open in $\omega$, halting if it is closed and otherwise exploring the cluster of the leftmost endpoint of $e$ in the halfspace lying strictly to the left of the rightmost endpoint of $e$. See the proof of \cite[Proposition 3.1]{Hutc20} to see how such a decision forest may be defined formally.
 This decision forest clearly computes $g$.

 Its revealments satisfy
\begin{align}
\delta_{e,\mathrm{perc}}(F,\P_{p,h}) &\leq \P_{p,h}(\text{$e\in \cG$ or at least one of the endpoints of $e$ has a pioneer in $\cG$})
\end{align}
\begin{align}
\delta_{e,\mathrm{ghost}}(F,\P_{p,h}) &= 1
\end{align}
for each $e\in \B$.
We can bound the revealment probabilities of edges by the union bound
\begin{align}
\delta_{e,\mathrm{perc}}(F,\P_{p,h}) &\leq \P_{p,h}(e\in \cG) + 2 \P_{p,h}(\text{$0$ has a pioneer in $\cG$})
\nnb & = 1-e^{-h}+ 2 \E_{p,h}\left[1-e^{-h |\cP_0|}\right].
\end{align}
It therefore follows from the OSSS inequality \cref{cor:OSSS_forest} that
\begin{align}
\Cov \left[f,g\right] &= \frac{1}{2}|\CoVr[f,g]|
\nnb
&\leq \sum_{e\in \B} \delta_{e,\mathrm{perc}}(F,\P_{p,h}) \Cov\left[f,\omega(e)\right] + \sum_{e\in \B} \delta_{e,\mathrm{ghost}}(F,\P_{p,h}) \Cov\left[f,\cG(e)\right]
\nonumber
\\
&=
\sum_{e\in \B} \delta_{e,\mathrm{perc}}(F,\P_{p,h}) \Cov\left[f,\omega(e)\right] \nnb
&\leq \left(1-e^{-h} + 2 \E_{p,h}\left[1-e^{-h |\cP_0|}\right]\right) \sum_{e\in \B} \Cov\left[f,\omega(e)\right],
\end{align}
where we used that $f(\omega,\cG)=f(\omega)$ is independent of the ghost field $\cG$ in the equality on the second line.
On the other hand, we can also compute that
\begin{align}
\Cov[f,g]
 &= \P_{p,h} \left(|\cP_0| \geq k, \; |\cP_0 \cap \cG| \geq 1 \right)  -  \P_p \left(|\cP_0| \geq k \right) \P_{p,h}\left(|\cP_0 \cap \cG| \geq 1\right)\nonumber
\\
&= \E_{p} \left[\left(1-e^{-h|\cP_0|}\right)\mathbbm{1}(|\cP_0| \geq k)\right]  - \E_{p} \left[1-e^{-h |\cP_0|}\right] \P_p\left(|\cP_0| \geq k\right)\nonumber\\
 &\geq (1-e^{-hk}) \P_p\left(|\cP_0| \geq k\right) - \E_p \left[1-e^{-h|\cP_0|}\right]
\P_p\left(|\cP_0| \geq k\right),
\end{align}
so that
\begin{equation}
\sum_{e\in \B}\Cov\left[f,\omega(e)\right] \geq \frac{(1-e^{-hk}) - \E_p \left[1-e^{-h|\cP_0|}\right] }{1-e^{-h} + 2 \E_p\left[1-e^{-h|\cP_0|}\right]}\P_p\left(|\cP_0| \geq k\right)
\end{equation}
for every $k\geq 1$, $0 \leq p \leq 1$ and $h \geq 0$.
The claimed
inequality \eqref{e:di-goal}
follows by taking $h=1/k$
and using the elementary fact that
\begin{equation}1-e^{-1/k} + \E_{p}\left[1-e^{-|\cP_0|/k}\right]
\le
\frac{1}{k}+\frac{1}{k}\E_{p}\left[\min\{k,|\cP_0|\}\right]=\frac{1}{k} \sum_{i=0}^k \P_p(|\cP_0| \geq i).
\end{equation}
This completes the proof.
\end{proof}

\begin{proof}[Proof of \cref{prop:subcritical_pioneers}]
We now analyse the differential inequality \eqref{eq:pioneersdiffineq} to prove the desired slightly subcritical bounds.
We begin with the proof of the inequality $\E_{p}
|\cP_0| \preceq  (p_c-p)^{-1/2}$.
Since $\P_p(|\cP_0| \geq k)$ is an increasing
function of $p$, we have by \cref{lem:differentialinequality} and
\cref{lem:criticalpioneers} that there exist positive constants $c_1$ and $C_1$ such that
\begin{equation}
\lrDini{p} \log \P_p(|\cP_0| \geq k) \geq  \frac{1}{2p(1-p)}
\left[\frac{k(1-e^{-1})}{C_1\sum_{i=0}^{k} (i+1)^{-2/3}}-1\right]
\geq  \frac{1}{2p(1-p)}
\left[c_1 k^{2/3} -1\right]
\end{equation}
for every $0<p \leq p_c$ and $k\geq 1$.
Integration of this inequality over the interval $[p,p_c]$, together with
\eqref{e:DiniFTC}, shows
that there exist positive constants $c_2$ and $C_2$ such that
\begin{align}
\P_p(|\cP_0| \geq k) &\leq \P_{p_c}(|\cP_0| \geq k) \exp\left(-\int_p^{p_c} \frac{1}{2q(1-q)} \left[c_1 k^{2/3} -1\right] d q \right)
\nonumber
\\
&\leq C_2 k^{-2/3} \exp\left(-c_2 (p_c-p) k^{2/3}\right)
\end{align}
for every $p_c/2 \leq p \leq p_c$ and $k \geq 1$. It follows by calculus that there exists a positive constant $C_3$ such that
\begin{equation}
\E_p |\cP_0| \leq  C_2 \sum_{k=1}^\infty k^{-2/3} e^{-c_2 (p_c-p) k^{2/3}} \leq C_3 (p_c-p)^{-1/2}
\end{equation}
for every $p_c/2\leq p < p_c$ as claimed.

Finally, we prove that $\P_{p}(
|\cP_0| \geq k) \preceq  k^{-2/3} \exp [ -c (p_c-p)^{3/2} k ]$.
The case of $p=p_c$ has been proved already in Lemma~\ref{lem:criticalpioneers},
so we can restrict attention here to $p_c/2\le p < p_c$.
The differential inequality \eqref{eq:pioneersdiffineq} implies the simplified inequality
\begin{equation}
\label{eq:pioneersdiffineq2}
\lrDini{p} \log \P_p(|\cP_0| \geq k) \geq  \frac{1}{2p(1-p)}
\left[\frac{k(1-e^{-1})}{1+\E_p|\cP_0|}-1\right].
\end{equation}
We again integrate the above inequality, and
conclude that there exist positive constants $c_3$, $c_4$ and $C_3$ such that if $0<\eps \leq p_c/4$ then
\begin{align}
\P_{p_c-2\eps}(|\cP_0| \geq k) &\leq \P_{p_c-\eps}(|\cP_0| \geq k) \exp\left( -\int_{p_c-2\eps}^{p_c-\eps} \frac{1}{2q(1-q)}
\left[c_3 \eps^{1/2} k-1\right] d q\right)
\nonumber
\\
&\leq C_3 k^{-2/3} \exp\left(-c_4\eps^{3/2}k\right)
\end{align}
for every $k\geq 1$. This completes the proof.
\end{proof}

\section{
Plateau below the window:  Proof of Theorem~\ref{thm:plateau}
}
\label{sec:below_window}

In this section, we apply our bound on the slightly subcritical $\Z^d$ two-point function from \cref{thm:2pt} to prove the plateau estimates for the torus two-point function
below the scaling window in \cref{thm:plateau}.
As a corollary, we also prove the torus triangle condition, \cref{thm:tricon}.

\subsection{Preliminaries}

We start by recording some preliminary estimates that we will use repeatedly in the rest of the section.
For each $x \in \Z^d$ and $0\leq p \leq 1$,
the $\Z^d$ \emph{open bubble} and \emph{open triangle diagrams} $\bubble_p(x)$
and $\triangle_p(x)$ are defined by
\begin{align}
    \bubble_p(x) &= \sum\limits_{u\in \Z^d}
		\tau_p(u) \tau_p(x-u)
    =
    (\tau_p*\tau_p)(x),
\\
	\triangle_p(x) &= \sum\limits_{u,v \in \Z^d}
		\tau_p(u) \tau_p(v-u)\tau_p(x-v)
    =
    (\tau_p*\tau_p*\tau_p)(x).	
\end{align}
Upper bounds on these two quantities are given in the next lemma.
The notation $p_c$ always refers to the critical value for $\Z^d$.

\begin{lemma}
\label{lem:tri_bubble_bounds}
Let $d>6$ and suppose that \eqref{eq:twopointassumption} holds on $\Z^d$.
There exist positive constants $C_1,C_2$ such that
\begin{align}
	\bubble_p(x)
	&\leq \frac{C_1}{\langle x \rangle^{d-4}}e^{-c_1m(p)\|x\|_\infty} , \\
	\triangle_p(x)
	&\leq \frac{C_2}{\langle x \rangle^{d-6}}e^{-c_1m(p)\|x\|_\infty}  ,
\label{e:Tbd}
\end{align}
for every $x \in \Z^d$ and every $p \le p_c$.
\end{lemma}

\begin{proof}
We insert the bound of \cref{thm:2pt} into the convolutions defining the bubble and
triangle diagrams.  By the triangle inequality, the exponential factors are bounded above
by an overall factor $e^{-c_1m(p)|x|}$.  For the powers, let
$f(x) = \langle x\rangle^{-(d-2)}$.  Since $d>6$ we have by
\eqref{e:ab_convolution}--\eqref{e:abc_convolution}  that
$(f*f)(x) \preceq \langle x \rangle^{-(d-4)}$ and $(f*f*f)(x) \preceq \langle x \rangle^{-(d-6)}$.
Together, this gives the desired result.
\end{proof}

Observe that
if $x \in \T_r^d$ is regarded as a point in $[-\frac r2, \frac r2)^d \cap \Z^d$
then $\langle x+ru \rangle \asymp r\langle u\rangle$ uniformly in nonzero $u\in \Z^d$ since
\begin{align}
\label{e:xulb}
    \|x+ru\|_\infty \ge \|ru\|_\infty - \frac r2 &\ge  \|ru\|_\infty - \frac 12\|ru\|_\infty
    = \frac 12 \|ru\|_\infty
\end{align}
and
\begin{align}
\label{e:xuub}
    \|x+ru\|_\infty \le  \frac r2 + \|ru\|_\infty  &\le  \frac 12\|ru\|_\infty + \|ru\|_\infty
    = \frac 32 \|ru\|_\infty.
\end{align}
The following elementary lemma will be useful with $\nu = cm(p)$.

\begin{lemma}
\label{lem:unifmassint}
Let $r \ge 2$, $a > 0 $ and $\nu >0$.
Then
\begin{align}
	\sum\limits_{u \in\Z^d : u \neq 0}
    \frac{1}{\|x + r u\|_\infty^{d-a}}e^{-\nu \|x+ru\|_\infty}
	& \preceq_a
     \frac{1}{\nu^{a}r^{d}} e^{-\frac 14 \nu r}
\end{align}
for every $x\in \T_r^d  \equiv [-\frac r2, \frac r2)^d \cap \Z^d$.
\end{lemma}

\begin{proof}
Let $a>0$. It follows
from
\eqref{e:xulb} that for any nonzero $u \in \Z^d$ and $r\geq 2$,
$\veee{x + r u} \geq \frac12 \|ru\|_\infty$ and thus
\begin{align}
	\sum\limits_{u \neq 0} \frac{1}{\|x + r u\|_\infty^{d-a}}e^{-\nu \veee{x+ru}}
	&\leq
    \sum\limits_{u \neq 0} \frac{1}{  (\frac 12 \|ru\|_\infty)^{d-a}}
    e^{-\frac{1}{2} \nu \|ru\|_\infty}\nnb
	&\le 2^{d-a}e^{-\frac 14 \nu r}\sum\limits_{N = 1}^\infty \sum\limits_{u:\|u\|_\infty =N}
    \frac{1}{\|ru\|_\infty^{d-a}}e^{-\frac 14 \nu \|ru\|_\infty}
    \nnb
    &\preceq_a r^{a-d}e^{-\frac 14 \nu r}
    \sum\limits_{N = 1}^\infty N^{d-1 - d + a }e^{-\frac 14 \nu rN}
	.
\end{align}
We bound the sum on the right-hand side by an integral to obtain an upper bound
which is a constant multiple of
\begin{align}
	r^{a-d}e^{-\frac 14 \nu r}
    \int_1^\infty u^{a-1} e^{-\frac 14 \nu ru} \D u
    & =
    \frac{1}{\nu^a r^d} e^{-\frac 14 \nu r} \int_{\nu r}^\infty t^{a-1}e^{-t/4} \D t
    .
\end{align}
The integral is uniformly bounded since $a>0$.
This concludes the proof.
\end{proof}

\begin{rk}
\label{rk:mchi}
Bounds expressed in terms of the mass $m(p)$, such as the one in Lemma~\ref{lem:unifmassint} with $\nu  = cm(p)$,
can also be expressed in terms of the
susceptibility $\chi(p)$ since
\begin{equation}
\label{e:mchi}
    \frac{1}{m(p)^2}  \preceq  \chi(p).
\end{equation}
To prove \eqref{e:mchi}, we first
fix any $p_1 \in (0,p_c)$.  For $p \le p_1$, since $m$ is decreasing and
since $1=\chi(0) \le \chi(p)$, we have $m(p)^{-2} \le m(p_1)^{-2} \le
m(p_1)^{-2}\chi(p)$ and the desired upper bound
follows for $p \in (0,p_1]$.
We can choose $p_1$ close enough to $p_c$ that $m(p)^{-2}$ and $\chi(p)$ are
comparable for $p\in (p_1,p_c)$, since both are asymptotic to $(1-p/p_c)^{-1}$. In particular
for $p_1$ close enough to $p_c$ there exists $C$ such that $m(p)^{-2} \leq C \chi(p)$ for those $p\in[p_1,p_c)$.
\end{rk}

The following three estimates will be useful.

\begin{lemma}
\label{lem:integr_tri_bubble_bounds}
Let $d>6$ and suppose that \eqref{eq:twopointassumption} holds on $\Z^d$.
For $r \ge 2$, $x\in \T^d_r$
and $0 \le p < p_c$,
\begin{align}
\label{e:plateau-ub}
	\sum_{u\in \Z^d}\tau_p(x+ru) &\leq \tau_p(x) + C\frac{\chi(p)}{V}e^{-\frac{c}{4} m(p)r}, \\
	\sum_{u\in\Z^d} \bubble_p(x+ru) &\leq \bubble_p(x) + C\frac{\chi(p)^2}{V}e^{-\frac{c}{4} m(p)r},  \\
	\sum_{u\in\Z^d} \triangle_p(x+ru) &\leq \triangle_p(x) + C\frac{\chi(p)^3}{V}e^{-\frac{c}{4} m(p)r}.
\end{align}
\end{lemma}

\begin{proof}
For the first inequality, we separate the $u = 0$ term from the sum
and apply \cref{thm:2pt}, as well as Lemma~\ref{lem:unifmassint}
with $a=2$, to obtain that there exist positive constants $c$, $C_1$, and $C_2$ such that
\begin{align}
	\sum_{u\in \Z^d}\tau_p(x+ru) &\leq  \tau_p(x) + \sum_{u \neq 0}
	\frac{C_1}{\langle x+ru\rangle^{d-2}}e^{-c m(p)\|x+ru\|_\infty} \\
	&\leq \tau_p(x) + C_2\frac{\chi(p)}{V}e^{-\frac{c}{4} m(p)r}.
\end{align}
For the bubble and triangle diagrams, in place of \cref{thm:2pt} we instead use
the bounds of Lemma~\ref{lem:tri_bubble_bounds}, which modify the power $d-2$ in the
above inequality to $d-4$ for the bubble and $d-6$ for the triangle.
We then apply Lemma~\ref{lem:unifmassint} and Remark~\ref{rk:mchi} with $a=4$ and with $a=6$ to complete
the proof.
\end{proof}

\subsection{Upper bound on the torus two-point function}

\begin{proof}[Proof of \eqref{e:plateau_upper}]
The proof is as in \cite{Slad20_wsaw}.
For each $0\leq p <p_c$ and $x\in \T^d_r \equiv [-\frac r2, \frac r2)^d \cap \Z^d$ we define
\begin{equation}
\label{e:psidef}
    \psi_{r,p}(x) = \sum_{u\in \Z^d: u \neq 0} \tau_p(x+ru),
\end{equation}
which is finite only when $p<p_c$. It follows by a simple coupling argument
originating in the work of
Benjamini and Schramm \cite{BS96} and further developed
in \cite[Proposition~2.1]{HHI07} that
\begin{equation}
\label{eq:2ptcoupling}
    \tau^{\T}_p(x) \le \tau_p(x) + \psi_{r,p}(x)
\end{equation}
for every $r>2$, $x\in \T^d_r$ and $0\leq p\leq 1$.
By \cref{lem:integr_tri_bubble_bounds},
\begin{align}
\label{e:plateau-ub2}
    \psi_{r,p}(x)
    &
	\leq C \frac{\chi(p)}{V}
    e^{-\frac{c}{4} m(p)r},
\end{align}
and with \eqref{eq:2ptcoupling} this immediately yields the upper bound \eqref{e:plateau_upper}.
\end{proof}

\subsection{Lower bound on the torus two-point function below the window}

We now turn to the proof of the lower bound \eqref{e:plateau_lower}
on the torus two-point function for $p$ below the scaling window.
This proof is model dependent and although it follows the general strategy
used for weakly self-avoiding walk in \cite{Slad20_wsaw}, it differs
in details.

We seek a lower bound of the form $r^{-d}\chi$ for the difference
\begin{equation}
    \psi_{r,p}^{\T}(x) =
    \tau^{\T}_{p}(x) - \tau_{p}(x)
    \qquad (x \in \T^d)
    .
\end{equation}
To this end we first make the decomposition
\begin{equation}
\label{e:psidecomp}
    \psi_{r,p}^{\T}(x) =  \psi_{r,p}(x) - (\psi_{r,p}(x) -\psi^{\T}_{r,p}(x) ).
\end{equation}
We can then deduce the lower bound \eqref{e:plateau_lower} as an immediate consequence of the following two lemmas.

\begin{lemma}
\label{lem:psigoal1}
Let $d>6$ and suppose that \eqref{eq:twopointassumption} holds on $\Z^d$.
There exist positive constants $A_2$ and $c_\psi$ such that if $r > 2$ and
$p_c-A_2r^{-2} \le p < p_c$
then
\begin{equation}
\label{e:psi}
    \psi_{r,p}(x) \ge c_\psi \frac{\chi(p)}{V}
\end{equation}
for every  $x\in \T^d_r$.
\end{lemma}

\begin{lemma}
\label{lem:psigoal2}
Let $d>6$ and suppose that \eqref{eq:twopointassumption} holds on $\Z^d$.
Let $c_\psi$ be as in \cref{lem:psigoal1}. There exist positive constants $A_1$ and $M$ such that if $r>2$ and $0 \le p \le p_c-A_1 V^{-1/3}$
then
\begin{equation}
\label{e:psigoal2-a}
    \psi_{r,p}(x) -\psi^{\T}_{r,p}(x)
    \le
    \frac 12 c_\psi \frac{\chi(p)}{V}
\end{equation}
for every $x\in \T^d_r$ with $\| x \|_\infty \geq M$.
\end{lemma}

\begin{proof}[Proof of \eqref{e:plateau_lower} subject to \cref{lem:psigoal1,lem:psigoal2}]
By definition,
\begin{equation}
    \tau_p^\T(x) = \tau_p(x) + \psi_{r,p}(x) - [ \psi_{r,p}(x) -\psi^{\T}_{r,p}(x)].
\end{equation}
By \cref{lem:psigoal1,lem:psigoal2}, if $\|x\|_\infty \ge M$ and $p_c-A_2r^{-2} \le p
\le p_c-A_1 V^{-1/3}$, then we have the lower bound
\begin{equation}
    \tau_p^\T(x)
    \ge
    \tau_p(x) + c_\psi \frac{\chi(p)}{V} - \frac 12 c_\psi \frac{\chi(p)}{V}
    =
    \tau_p(x) +  \frac 12 c_\psi \frac{\chi(p)}{V},
\end{equation}
which is the desired estimate.
\end{proof}

\subsubsection{Proof of \cref{lem:psigoal1}}

In this section we prove the lower bound on $\psi_{r,p}(x)$ stated in
\cref{lem:psigoal1}.
We begin with the following simple observation.

\begin{lemma}
\label{lem:Gamma}
Let $d>6$ and suppose that \eqref{eq:twopointassumption} holds on $\Z^d$.
The inequality
\begin{align}
\label{e:Gamma4}
    \tau_{p_c}(x)-\tau_p(x) & \preceq
    \frac{p_c-p}
    {\langle x\rangle^{d-4}}.
\end{align}
holds for every $p_c/2 \le  p \le p_c$ and $x\in \Z^d$.
\end{lemma}

\begin{proof}
The proof of \eqref{e:Gamma4} is a consequence of the following standard differential
inequality (cf. \cite{AN84}).
Let $\tau_p^n(x) = \P(0 \conn x \text{ inside }[-n,n]^d)$. It is easy to see that for every $n>0$, $\tau^n_p(x)$ is differentiable in $p$ and that $\tau_p^n(x) \to \tau_p(x)$ as $n \to \infty$.
Let $p\in [\frac 12 p_c,p_c]$.
By Russo's Formula and the BK inequality
(with the sum over the undirected bonds in $[-n,n]^d$)
we have
\begin{align}
    \frac{d}{dp}\tau^n_p(x)	&= \frac1p\sum\limits_{\{u,v\}}\P_{p}(\{u,v\} \text{ is pivotal for } 0 \conn x\text{ inside }[-n,n]^d,\, \{u,v\} \text{ is open}) \nnb
    						&\leq \frac1p\sum\limits_{\{u,v\}}\P_{p}(\{0 \conn u \text{ inside }[-n,n]^d\} \circ \{u \conn x\text{ inside }[-n,n]^d\}) \nnb
    						\label{e:convest0}
    						&  \preceq
                                (\tau_p*\tau_p)(x).
\end{align}
It follows by monotonicity in $p$ and the
bound on the bubble from \cref{lem:tri_bubble_bounds} that
\begin{equation}
\label{e:convest}
    \frac{d}{dp}\tau^n_p(x)
    \preceq (\tau_{p_c}*\tau_{p_c})(x)
    \preceq \frac{1}{\langle x\rangle^{d-4}}.
\end{equation}
Integration of \eqref{e:convest} over $[p,p_c]$, followed by
the limit as $n\to \infty$, gives \eqref{e:Gamma4}.
\end{proof}

We now apply \cref{lem:Gamma} to complete the proof of \cref{lem:psigoal1}.

\begin{proof}[Proof of \cref{lem:psigoal1}]
Let $x \in \T_r^d$.
To obtain a lower bound on $\psi_{r,p}$,
we may sum in \eqref{e:psidef}
over only those $u\in \Z^d$ with $\|u\|_\infty \le R$ with $R\geq 1$ a large number depending on $r$ and $p_c-p$
to be chosen shortly.
By \eqref{eq:twopointassumption} and \eqref{e:Gamma4}, there exist positive constants $c_1$ and $C_1$ such that, for every $y \in \Z^d$,
\begin{align}
\label{eq:tau_decomp}
    \tau_p(y) & = \tau_{p_c}(y) - \big(\tau_{p_c}(y) - \tau_p(y)\big)
    \ge
    \frac{c_1}{\langle y\rangle^{d-2}} -  \frac{C_1 (p_c-p)}{\langle y\rangle^{d-4}}  .
\end{align}
With this,
together with \eqref{e:xulb}--\eqref{e:xuub} we see
that there exist positive constants $c_2, C_2,C_3$ such that
\begin{align}
    \sum_{u\in \Z^d: u \neq 0} \tau_p(x+ru) &\ge
    \sum_{1\leq  \|u\|_\infty \le R} \tau_p(x+ru)
    \nnb & \ge
    \frac 23 c_1\sum_{1\leq \| u\|_\infty\le R} \frac{1}{\|ru\|_\infty^{d-2}}
    - 2C_1 (p_c-p)\sum_{1\leq \|u\|_\infty \le R} \frac{1}{\| ru\|_\infty^{d-4}}
    \nnb &\ge
    \frac{c_2}{r^{d-2}}R^2 -  \frac{C_2 (p_c-p)}{r^{d-4}} R^4
     =
    \frac{c_2}{r^{d-2}}R^2 \left(1 - C_3(p_c-p)  r^2 R^2 \right)
\end{align}
for every $r,R\geq 1$ and $p_c/2 \leq p <p_c$.
Now we choose $R^2=(2C_3 (p_c-p) r^{2})^{-1}$, and require $p_c-p \leq A_2 r^{-2}$ with
$A_2$ chosen small enough for $R$ to be indeed greater than $1$.  This gives
\begin{align}
    \sum_{u\in \Z^d: u \neq 0} \tau_p(x+ru)
    & \ge
    \frac{c_2 R^2}{2r^{d-2}} =
    \frac{c_3}{(p_c-p)r^d} \succeq \frac{\chi(p)}{V}
\end{align}
for every $r\geq 2$ and
$p \in [p_c-  A_2 r^{-2}, p_c)$,
and completes the proof.
\end{proof}

\subsubsection{Proof of \cref{lem:psigoal2}}

We now prove \cref{lem:psigoal2}, which states that there exist constants $M$ and $A_1$ such that if
$p \le p_c-A_1V^{-1/3}$
then
\begin{equation}
\label{e:psigoal}
    \psi_{r,p}(x) -\psi^{\T}_{r,p}(x)
    \le
    \frac 12 c_\psi r^{-d}\chi(p)
\end{equation}
for every $x\in \T^d_r$ with $\|x\|_\infty \ge M$,
where $c_\psi$ is the constant from \cref{lem:psigoal1}.
 In order to prove this, we will prove the following more general inequality.

\begin{lemma}
\label{lem:psigoal2_diagram}
Let $d>6$ and suppose that \eqref{eq:twopointassumption} holds on $\Z^d$.
There exists a constant $C$ such that the inequality
\begin{align}
\label{e:psigoal2-b}
	\psi_{r,p}(x) - \psi_{r,p}^\T(x)
    &\leq C\frac{\chi(p)}{V}
    \left(\triangle_p(x)
    +\frac{\chi(p)^3}{V} \right),
\end{align}
holds for every  $r>2$, $x\in \T^d_r$ and $p<p_c$.
\end{lemma}

\begin{proof}[Proof of \cref{lem:psigoal2} given \cref{lem:psigoal2_diagram}]
By taking $p \le p_c - A_1 V^{-1/3}$, we see from the bound on the susceptibility
in \eqref{e:chiasy} that $\chi(p) \preceq A_1^{-1} V^{1/3}$, so the term $V^{-1}\chi(p)^3$
can be made as small as desired by taking $A_1$ sufficiently large.
By \eqref{e:Tbd}, the triangle term $\triangle_p(x)$ can be made as small as desired
by taking $\|x\|_\infty \geq M$ with $M$ sufficiently large.
Thus we can choose the constants $A_1,M$ in such a way that the right-hand side of
\eqref{e:psigoal2-b} is at most $\frac 12 c_\psi \chi(p)/V$.
This gives the desired inequality \eqref{e:psigoal}.
 \end{proof}

We turn now to the proof of Lemma~\ref{lem:psigoal2_diagram}.
We build upon the coupling of percolation on $\Z^d$ and $\T^d_r$ developed by
Heydenreich and van der Hofstad  \cite[Proposition 2.1]{HHI07}.
  With this coupling, they proved that
\begin{align}
\label{eq:HHI07}
    \psi_{r,p}(x) -\psi^{\T}_{r,p}(x)
     \le
    \frac 12 \sum_{u\in \Z^d}\sum_{v\neq u}\P(0 \conn x+ru,\, x+rv)
    + \sum_{u \in\Z^d} \P(\{0 \conn x+ru\} \cap \{0 \xleftrightarrow[\T]{} x\}^c)
\end{align}
for every $x\in \T^d_r$, where $\{x \xleftrightarrow[\T]{} y\}$
denotes the event that $x$ is connected to $y$ by an open path in $\T^d_r$ in the coupling
(see \cite[(5.4)]{HHI07}).

The proof of \cref{lem:psigoal2_diagram}
is immediate using the  following two lemmas to bound the two terms in \eqref{eq:HHI07}.
In the first lemma, there is
room to spare by a factor $\chi$ in the last term; this is consistent with
\cite{HHI07}.  Also we see the bubble rather than the triangle, which again has room to spare.

\begin{lemma}
\label{lem:1st-term}
Let $d>6$ and suppose that \eqref{eq:twopointassumption} holds on $\Z^d$.
The inequality
\begin{equation}
    \sum_{u\in \Z^d}\sum_{v\neq u}\P(0 \conn x+ru,\, x+rv)
    \preceq
    \frac{\chi}{V} \left(\bubble_p(x) + \frac{\chi^2}{V} \right)
\end{equation}
holds for every $0\leq p <p_c$, $r >
2$, and $x\in \T^d_r$.
\end{lemma}

\begin{proof}
We use $x,y$ for torus points and $u,v,w$ for translating points in $\Z^d$,
and for clarity write the two-point function as
$\tau(u,v)$
in place of the usual
$\tau_p(v-u)$.
By the BK inequality,
\begin{align}
    &\sum_{u\in \Z^d}\sum_{v\neq u}\P(0 \conn x+ru,\, x+rv)
    \nnb & \qquad \le
    \sum_{z,u\in \Z^d}\sum_{v\neq u} \tau(0,z)\tau(z,x+ru) \tau(z,x+rv)
    \nnb
    & \qquad
    =
    \sum_{y\in\T_r^d} \sum_{w\in\Z^d}
    \tau(0,x+y+rw)\sum_{u\in\Z^d} \tau(y,r(u-w))
    \sum_{v\neq u}\tau(y,r(v-w))
    \nnb
    & \qquad
    =
    \sum_{y\in\T_r^d} \sum_{w\in\Z^d}
    \tau(0,x+y+rw)\sum_{u\in\Z^d} \tau(y,ru)
    \sum_{v\neq u}\tau(y,rv)
    ,
\label{e:1st-term}
\end{align}
where in the third line we replaced $z$ by $x+y+rw$, and in the fourth
we replaced $u$ by $u+w$ and $v$ by $v+w$.
For the sum over $v$, it follows from Lemma~\ref{lem:integr_tri_bubble_bounds} that
\begin{align}
    \sum_{v\neq u}\tau(y,rv)
    & =
    \sum_{v\neq u}\tau(y,rv) (\1_{u=0}+\1_{u\neq 0})
    \nnb & =
    \1_{u=0}\sum_{v\neq 0}\tau(y,rv) +\1_{u\neq 0}\sum_{v\neq u}\tau(y,rv)
    \nnb & \preceq
    \1_{u=0} \frac{\chi}{V} + \1_{u\neq 0}\left(\tau(0,y)+ \frac{\chi}{V} \right)
    \preceq
    \1_{u\neq 0} \tau(0,y) + \frac{\chi}{V}
    .
\end{align}
This leads, using Lemma~\ref{lem:integr_tri_bubble_bounds} again, to
\begin{align}
    \sum_{u\in\Z^d} \tau(y,ru)
    \sum_{v\neq u}\tau(y,rv)
    & \preceq
    \tau(0,y) \frac{\chi}{V} + \frac{\chi}{V} \left(\tau(0,y)+ \frac{\chi}{V} \right)
 \preceq
    \frac{\chi}{V} \left(\tau(0,y)+ \frac{\chi}{V} \right)
    .
\end{align}
Thus we have an upper bound on \eqref{e:1st-term} given by
\begin{align}
    \frac{\chi}{V}
    \sum_{y\in\T_r^d} \sum_{w\in\Z^d}
    \tau(0,x+y+rw)
    \left(\tau(0,y)+ \frac{\chi}{V} \right)
     =
    \frac{\chi^3}{V^2} + \frac{\chi}{V}
    \sum_{y\in\T_r^d} \sum_{w\in\Z^d}
    \tau(0,x+y+rw) \tau(0,y)
    .
\end{align}
We extend this last sum over $y$ to all of $\Z^d$ and
use the inequality for the bubble from Lemma~\ref{lem:integr_tri_bubble_bounds} to finally get that
\begin{equation}
    \sum_{u\in \Z^d}\sum_{v\neq u}\P(0 \conn x+ru,\, x+rv)
     \preceq \frac{\chi}{V} \left(\bubble_p(x) + \frac{\chi^2}{V} \right)
\end{equation}
as claimed.
\end{proof}

\begin{lemma}
\label{lem:2ndterm}
Let $d>6$ and suppose that \eqref{eq:twopointassumption} holds on $\Z^d$.
The estimate
\begin{equation}
     \sum_{u \in\Z^d} \P(\{0 \conn x+ru\} \cap \{0 \xleftrightarrow[\T]{} x\}^c)
    \preceq \frac{\chi(p)}{V} \left( \triangle_p(x) + \frac{\chi(p)^3}{V} \right).
\end{equation}
holds for every $0\leq p <p_c$, $r> 2$, and $x\in \T^d_r$.
\end{lemma}

\begin{proof}
Our starting point is the set inclusion
\begin{multline}
\label{eq:HHI07diagram}
\{0 \conn x+ru\} \cap \{0 \xleftrightarrow[\T]{} x\}^c \subseteq
 \bigcup\limits_{z\in \Z^d}\bigcup\limits_{a \in \T_r^d}\bigcup\limits_{v_1,v_2 \in \Z^d: v_1 \neq v_2} \{0\conn z\}\circ \{z\conn a+rv_1\}\\
 \circ \{z\conn a+rv_2\}\circ\{a+rv_2\conn x+ru\}
\end{multline}
for every $x\in \T^d_r$ and $u \in \Z^d$, which arises from
the coupling of torus and $\Z^d$ percolation in \cite[Proposition~2.1]{HHI07}.
We use $a,b,x,y,z$ for torus points and use $u,v,w$ for translating vectors.
It follows from the set inclusion \eqref{eq:HHI07diagram} together with a union bound and the BK inequality that
\begin{multline}
\P(\{0 \conn x+ru\} \cap \{0 \xleftrightarrow[\T]{} x\}^c) \\\leq
	\sum\limits_{a \in \T_r^d}
	\sum\limits_{z,v_1 \in \Z^d}
	\sum\limits_{v_2 \neq v_1}
	\tau_p(z)\tau_p(a+rv_2 - z)\tau_p(a+rv_1 - z)\tau_p(x+ru -a-rv_2).
\end{multline}

We translate to more convenient vertices, as follows.
First, we write the $\Z^d$ point $a+rv_2-x$ uniquely as a torus point $y$ plus $rv$
with $v\in \Z^d$, and similarly for the others, to obtain
\begin{align}
	a+ rv_2 -x &= y + rv,  & y\in \T_r^d, \quad v \in \Z^d,\nnb
	a+r v_1 -x &= y + rv +rv',& v'=v_2-v_1 \neq 0,  \\ \nonumber	
	z -x&= y+z' +r u',& z'\in \T_r^d,\quad u' \in \Z^d.
\end{align}
This gives
\begin{multline}
\sum_{u\in \Z^d}\P(\{0 \conn x+ru\} \cap \{0 \xleftrightarrow[\T]{} x\}^c)
\leq \sum\limits_{y,z' \in \T_r^d}\sum\limits_{v,u'\in\Z^d}
\tau_p(x+y+z'+ru')\tau_p( -z'+r(v - u'))
\\ \times
\sum\limits_{v' \neq 0}\tau_p(-z'+r(v'+v - u'))\sum\limits_{u\in\Z^d}\tau_p(-y + r(u-v)).
\end{multline}
We bound the sums over $u$ and $v'$ with Lemma~\ref{lem:integr_tri_bubble_bounds} and obtain
\begin{align}
	\sum\limits_{u\in\Z^d}\tau_p(-y + r(u-v)) =\sum\limits_{u\in\Z^d}\tau_p(-y + ru) &\leq \tau_p(y) + C\frac{\chi(p)}{V} , \\
	\sum\limits_{v' \neq 0}\tau_p(-z'+r(v'+v - u'))=\sum\limits_{v' \neq v-u'}\tau_p(-z'+rv') &\leq \tau_p(z')\1_{v\neq {u'}} + C\frac{\chi(p)}{V}.
\end{align}
Then we perform the sum over $v$, which after translating $v$ by $u'$ is bounded similarly using
\begin{align}
&\sum_{v\in \Z^d}\tau_p( -z'+r(v-u'))
    \Big(\tau_p(z')\1_{v\neq u'} + C\frac{\chi(p)}{V}\Big)	
    \nnb &\qquad =\sum_{v\in \Z^d}\tau_p( -z'+rv)
    \Big(\tau_p(z')\1_{v\neq 0} + C\frac{\chi(p)}{V}\Big)
    \nnb
    &\qquad\leq C\frac{\chi(p)}{V}\tau_p(z') + C\frac{\chi(p)}{V}
    \sum\limits_{v \in \Z^d}\tau_p( -z'+rv) \nnb
	&\qquad\preceq \frac{\chi(p)}{V}\tau_p(z') + \frac{\chi(p)^2}{V^2}.
\end{align}
This leads to
\begin{multline}
\label{e:p_e}
\sum_{u\in \Z^d}\P(\{0 \conn x+ru\} \cap \{0 \xleftrightarrow[\T]{} x\}^c)\\\preceq
\frac{\chi(p)}{V}
\sum\limits_{y,z' \in \T_r^d}\sum_{u'\in\Z^d}
\tau_p(x+y+z'+ru')\left(\tau_p(y) + \frac{\chi(p)}{V}\right)
\left(\tau_p(z')+\frac{\chi(p)}{V}\right).
\end{multline}
We expand out the brackets and recognise that the term containing the product $\tau_p(y)\tau_p(z')$ obeys
\begin{align}
\label{e:diagr_tri_bound}
&\sum_{{u'\in\Z^d}}\sum\limits_{y,z' \in \T^d_r}\tau_p(x+y+z'+r u')\tau_p(z')\tau_p(y)
\nnb &\qquad
=\sum_{{u'\in\Z^d}}\sum\limits_{y,z' \in \T^d_r}\tau_p(z')\tau_p(y+z'-z') \tau_p(x+y+z'+r{u'})
\nonumber\\
&\qquad \leq \sum_{u'\in \Z^d} \triangle_p(x+ru')
\preceq  \triangle_p(x) + \frac{\chi(p)^3}{V}.
\end{align}
Meanwhile, the two terms containing exactly one of $\tau_p(y)$ or $\tau_p(z')$ are equal and can be expressed  as
\begin{align}
&\frac{\chi(p)}{V}\sum_{{u'\in\Z^d}}\sum\limits_{y,z' \in \T^d_r}\tau_p(x+y+z'+r{u'})\tau_p(z')
\nnb &\qquad
=\frac{\chi(p)}{V}\sum_{{w\in\Z^d}}\sum\limits_{z' \in \T^d_r}\tau_p(x+z'+w)\tau_p(z')
\leq \frac{\chi(p)^3}{V},
\end{align}
where we extended the sum over $z'$ to all of $\Z^d$ in the last inequality.
Finally, the term not containing either $\tau_p(y)$ or $\tau_p(z')$ can be expressed as
\begin{equation}
    \frac{\chi(p)^2}{V^2}\sum\limits_{y,z' \in \T_r^d}\sum_{{u'\in \Z^d}}
	\tau_p(x+y	+z'+ru') = \frac{\chi(p)^2}{V^2}\sum_{z' \in \T_r^d}\sum_{w\in \Z^d}
  \tau_p(x + z'+ w) =  \frac{\chi(p)^3}{V}.
\end{equation}
 Summation of these contributions gives
	\begin{equation}
	\sum_{u \in\Z^d} \P(\{0 \conn x+ru\} \cap \{0 \xleftrightarrow[\T]{} x\}^c)
    \preceq \frac{\chi(p)}{V} \Big(\triangle_p(x) + \frac{\chi(p)^3}{V}\Big),
\end{equation}
and the proof is complete.
\end{proof}

  \begin{proof}[Proof of \cref{lem:psigoal2_diagram}]
\cref{lem:1st-term,lem:2ndterm} give bounds on the two terms on the right-hand side
of \eqref{eq:HHI07}, namely
\begin{equation}
        \psi_{r,p}(x) -\psi^{\T}_{r,p}(x)
        \preceq
        \frac{\chi}{V} \left(\bubble_p(x) + \frac{\chi^2}{V} \right)
        +
        \frac{\chi}{V} \left( \triangle_p(x) + \frac{\chi^3}{V} \right).
\end{equation}
Since the bubble is bounded above by the triangle and the susceptibility is at least $1$,
this gives the desired estimate
\begin{equation}
        \psi_{r,p}(x) -\psi^{\T}_{r,p}(x)
        \preceq
        \frac{\chi}{V} \left( \triangle_p(x) + \frac{\chi^3}{V} \right)
\end{equation}
and the proof is complete.
  \end{proof}

\subsection{The torus triangle condition: Proof of Theorem~\ref{thm:tricon}}
\label{sec:triangle_cond}

To conclude this section,
we show how the torus plateau leads to easy proofs that $p_\T$ lies
in the scaling window and that the torus triangle condition holds.

\begin{proof}[Proof of Theorem~\ref{thm:tricon}]
Fix $\eps>0$ sufficiently small that $\eps^{-1} \geq A_2$, where $A_2$ is as in \cref{thm:plateau}. Recalling from \eqref{e:chiasy} that $\chi \asymp (p_c-p)^{-1}$ and setting $p_0=p_c-\eps^{-1}V^{-1/3}$, we have that $\chi(p_0) \asymp
\eps V^{1/3}$. On the other hand, for sufficiently large $r$
(depending on $M$) the
lower bound of \eqref{e:plateau_lower} applies to give
\begin{equation}
\label{e:chiTp0}
    \chi^{\T}(p_0)
    \succeq
    \sum_{x \in \T_r^d\colon\|x\|_\infty > M}  V^{-1}\chi(p_0)
    \succeq (V-(2M+1)^d) V^{-1}\chi(p_0) \succeq \eps V^{1/3}.
\end{equation}
Since we also have by the coupling that $\chi^\T \le \chi$ it follows that there exist positive constants $c_1$ and $C_2$ such that
\begin{equation}
    c_1 \eps V^{1/3} \leq \chi^{\T}(p_0) \le \chi(p_0) \leq C_2 \eps V^{1/3}.
\end{equation}
A second application of \eqref{e:chiasy} yields that there exists a constant $C_3 \geq 1$ such that if we define $p_1=p_c-C_3\eps^{-1}V^{-1/3}$ then $\chi^\T(p_1) \leq \chi(p_1) \leq c_1 \eps V^{1/3}$. It follows by the intermediate value theorem that if we define $\lambda = \lambda(\eps)= c_1 \eps$ then the
$p_\T$
defined by $\chi^{\T}(p_\T)=\lambda V^{1/3}$ (which does exist if $r$ exceeds some value
$r_0(\lambda)$)
satisfies $p_1 \leq p_\T \leq p_0$ and hence that
$0 \le p_c- p_\T
\preceq \eps^{-1} V^{-1/3}$.
With the choice $\lambda_0=c_1A_2^{-1}$, this
concludes the proof
that $p_\T=p_\T(\lambda)$ lies in the scaling window if $\lambda \in (0,\lambda_0]$
and $r > r_0(\lambda)$.

Let $p<p_c$ and $x\in \T_r^d$.
The open torus triangle diagram is defined by
\begin{align}
\label{e:torus-tri}
    \triangle^{\T}_{p}(x)
    & =
    \sum_{y,z\in \T_r^d} \tau^{\T}_p(y) \tau^{\T}_p(z-y) \tau^{\T}_p(x-z),
\end{align}
and \eqref{eq:2ptcoupling} then implies that
\begin{align}
\triangle^{\T}_{p}(x)
    & \leq \sum_{y,z \in \T_r^d}\sum_{u,v,w \in \Z^d}\tau_p(y+ru)\tau_p(z-y+rv)\tau_p(x-z+rw).
\end{align}
We replace the index $v$ by $v'-u$ and then replace $w$ by $w'-v'$.
The above right-hand side becomes (after setting $y'=y+ru$ and $z'=z+rv'$)
\begin{align}
\sum_{y',z',w' \in \Z^d}\tau_p(y')\tau_p(z'-y')\tau_p(x-z'+rw')
    &= \sum_{w' \in \Z^d}\triangle_p(x+rw'),
\end{align}
with $\triangle_p(x+rw')$ the open $\Z^d$ triangle diagram.
By Lemma~\ref{lem:integr_tri_bubble_bounds}, this gives that
\begin{equation}
\label{e:tor_tri_cond-x}
 \triangle^{\T}_{p}(x) \leq \triangle_p(x) + C_1\frac{\chi(p)^3}{V}
 \leq \triangle_{p_c}(x) + C_1\frac{\chi(p)^3}{V}.
\end{equation}
With $p_\T$ as in the previous paragraph,
this implies that
\begin{equation}
\label{e:tor_tri_cond-x2}
 \triangle^{\T}_{p_\T}(x) \leq  \triangle_{p_c}(x) + C_1 C_2^3 \eps^3.
\end{equation}
With the bound
$\eps\le A_2^{-1}$, this concludes the proof of the torus triangle condition.

It remains to prove that the $a_0$-strong version of the triangle condition holds when
its $\Z^d$ counterpart holds with $\frac 12 a_0$.
But under this assumption it follows from \eqref{e:tor_tri_cond-x2} that
\begin{equation}
\label{eq:torus-triangle-strong}
\triangle_{p_\T}^\T(x) \leq \mathbbm{1}(x=0)+ \frac{1}{2}a_0 + C_1C_2^3 \eps^3.
\end{equation}
We require that $\eps^3 \le \eps_0^3 = a_0(2C_1C_2^3)^{-1}$.  With $\lambda\le
\lambda_1 = c_1\eps_0$, the right-hand side of \eqref{eq:torus-triangle-strong}
is at most $\mathbbm{1}(x=0)+a_0$, and the proof is complete.
\end{proof}

\section{
Plateau within the scaling window:  Proof of Theorem~\ref{thm:plateau}
}
\label{sec:pcbd}

We now turn to the part of \cref{thm:plateau} concerning the case that $p$ lies within the scaling window of the torus.
The window consists of $p$ values with $|p-p_c| \le AV^{-1/3}$ with $A$ arbitrary
but fixed.

\subsection{Lower bound in the window}

We begin by proving the lower bound \eqref{e:plateaupc_lower}, which
follows simply from the monotonicity of $\tau_p^\T$ in $p$,
the lower bound below the window, and the comparison of $\tau_p$ with $\tau_{p_c}$
provided by \cref{lem:Gamma}.

\begin{proof}[Proof of \eqref{e:plateaupc_lower}]
Let $A_1$ and $A_2$ be the constants from the ``below the scaling window" part of \cref{thm:plateau}. Fix
$A>0$.
It suffices by monotonicity of the torus two-point function to prove the claimed estimate in the case that $A \geq A_1$
and $p=p_c- A  V^{-1/3}$.  We denote this value of $p$
by $p'$.
Thus, there exists a constant $r_0$ depending on $A$ such that
$A V^{-1/3} \leq A_2 r^{-2} = A_2 V^{-2/d} le p_c/2$ for every $r\geq r_0$.

By \eqref{e:chiasy}, $\chi(p') \ge c_\chi A^{-1}V^{1/3}$ for some constant $c_\chi$ depending
only on $d$ and $L$.
It then follows from \eqref{e:plateau_lower} that there exists a constant $M$ (depending on $d$ and $L$) such that
 \begin{equation}
 \label{e:taupprime}
    \tau^\T_{p'}(x) \geq \tau_{p'}(x)+\frac{c_2c_\chi}{A V^{2/3}}
\end{equation}
for every $x\in \T^d_r$ with $\|x\|_\infty \geq M$.
By \cref{lem:Gamma} (with our assumption that $r \ge r_0$ guaranteeing
that $p' \ge p_c/2$), we have moreover that
\begin{equation}
\label{e:tauTwindowlb}
   \tau_{p_c}(x) - \tau_{p'}(x)
    \preceq \frac{A}{V^{1/3}\langle x\rangle^{d-4}}  \preceq \frac{A r^2}{V^{1/3}} \tau_{p_c}(x).
\end{equation}
Since the prefactor $A r^2 V^{-1/3}$ tends to zero as $r\to\infty$, it follows from \eqref{e:taupprime}--\eqref{e:tauTwindowlb} that for each $\delta>0$ there exists
an $r_1 \geq r_0$ depending on $A$ such that
\begin{equation}
    \tau^\T_{p'}(x) \geq (1-\delta)\tau_{p_c}(x)+
    \frac{c_2c_\chi}{A V^{2/3}}
\end{equation}
for every $x\in \T^d_r$  with $\|x\|_\infty \geq M$ and every $r \ge r_1$.
This completes the proof.
\end{proof}

\subsection{Upper bound in the window}

It remains to prove the upper bound. To do so we first consider
the case $p=p_c$.  We then extend the upper bound to the window $(p_c,p_c+AV^{-1/3}]$ by proving that, in
this window, the two-point function changes only up to a multiplicative
factor (that can be chosen to be arbitrarily close to one) plus an additive constant term
of order $V^{-2/3}$.

At $p_c$, the upper bound is not new and was
proven previously in \cite[Theorem~1.7]{HS14}. However, our proof, which is based on
the extrinsic (Euclidean)
distance, seems more direct than that of
\cite{HS14} where the  intrinsic distance was used.
Our proof relies
on the extrinsic one-arm exponent estimate
\begin{equation}
\label{e:1-arm-ext}
\P_{p_c}(0 \conn \partial \Lambda_\ell) \asymp \frac{1}{\ell^2}
\end{equation}
of Kozma and Nachmias \cite{KN11} (i.e.\ the $p=p_c$ case of \cref{thm:main1arm}).

\begin{prop}
Let $d>6$ and suppose that \eqref{eq:twopointassumption} holds on $\Z^d$.
There is a $C>0$ such that
\label{prop:ub-at-p_c}
\begin{equation}
	\tau_{p_c}^\T(x) \leq \tau_{p_c}(x) + CV^{-2/3}
\end{equation}
for all $r>2$ and all $x\in\T_r^d$.
\end{prop}

\begin{proof}
Fix some large and positive integer $M$.
Since $\Z^d$ covers the torus $\T^d_r$, every path in $\T^d_r$ can be lifted to a path in $\Z^d$ that is unique up to the choice of starting point.
For $x,y \in \T_r^d$, we define $E_{\geq \ell}(x,y)$ to be the event that $x$ and $y$ are connected by a simple $\T_r^d$-path that lifts to a  $\Z^d$-path of diameter greater than or equal to $\ell$, and we define $E_{\leq \ell}(x,y)$ similarly. In addition, we define $A_{\geq \ell}(x)$ to be the event that there exists some simple $\T_r^d$-path starting from $x$ that lifts to a $\Z^d$-path of diameter at least $\ell$.
We have trivially that $\{0 \xleftrightarrow[\T]{} x\} = E_{\leq 3Mr}(0,x)\cup E_{\geq 3Mr}(0,x)$, so that
\begin{align}
\label{e:PMr}
		\P^\T_{p_c}(0 \xleftrightarrow[\T]{} x)
		&\leq \P^\T_{p_c}(E_{\leq 3Mr}(0,x))
		+ \P^\T_{p_c}(E_{\geq 3Mr}(0,x)).
\end{align}
Note that on $E_{\geq 3Mr}(0,x)$, the events $A_{\geq Mr}(0)$ and  $A_{\geq Mr}(x)$ must occur disjointly.
Also, in the coupling between torus and $\Z^d$ percolation, we have
\begin{align}
	A_{\geq \ell}(0) &\subset \{0 \xleftrightarrow[\Z]{} \partial \Lambda_{\ell}\} \qquad \text{and} \qquad
	A_{\geq \ell}(x) \subset \{x \xleftrightarrow[\Z]{}  x + \partial \Lambda_{\ell}\},
\end{align}
where we recall that $\Lambda_\ell=[-\ell,\ell]^d\cap \Z^d$. Thus, by the BK inequality on the torus,
\begin{align}
	\P^\T_{p_c}(E_{\geq 3Mr}(0,x))
	&\leq \P^\T_{p_c}(A_{\geq Mr}(0)\circ A_{\geq Mr}(x)) \nnb
	&\leq \P^\T_{p_c}(A_{Mr}(0))\P^\T_{p_c}(A_{Mr}(x))
    \leq \P_{p_c}(0 \conn \partial \Lambda_{Mr})^2.
\end{align}
Using the one-arm upper bound of \eqref{e:1-arm-ext}, this gives
\begin{align}
	\P^\T_{p_c}(E_{\geq 3Mr}(0,x))
	&\preceq \frac{1}{M^4r^4}.
\end{align}

For the first term of \eqref{e:PMr} we simply use the coupling and a union bound to see that
\begin{align}
	\P^\T_{p_c}(E_{\leq 3Mr}(0,x))
	&\leq \P_{p_c}(\cup_{\|u\|_\infty \leq 3M } \{0 {\conn} x+ru\})
    \leq \tau_{p_c}(x)+ \sum_{1\leq\|u\|_\infty \leq 3M}\tau_{p_c}(x+ru).
\end{align}
The latter sum can be bounded above  by an integral over the $d$-dimensional box of
radius $3M$, namely
\begin{equation}
     \sum_{1\leq\|u\|_\infty \leq 3M}\tau_{p_c}(x+ru)
     \preceq
     \int_{\|u\|_\infty \le 3M} \frac{1}{{\langle ru\rangle}^{d-2}} \D u \preceq M^2 r^{-(d-2)}.
\end{equation}
Together, these bounds imply that there exists a constant $C$ such that
\begin{equation}
	\tau_{p_c}^\T(x) \leq \tau_{p_c}(x) + CM^2r^{-(d-2)} + CM^{-4}r^{-4}.
\end{equation}
The choice $M = r^{(d-6)/6}$ gives the desired upper bound
$\tau_{p_c}(x) + CV^{-2/3}$ at $p = p_c$.
\end{proof}

Next, we prove an upper bound at the top of the scaling window.
For  $p\in (p_c,p_c+A V^{-1/3}]$ we use
the \emph{intrinsic distance} $d_\text{int}$, which is the graph distance on the percolation
configuration.  If $x,y$ are not connected in the configuration, then $d_\text{int}(x,y)=\infty$.
Given a percolation configuration, we define the (random) \emph{intrinsic ball}
centred at $x$ and of radius $\ell$ by
\begin{equation}
	B_\mathrm{int}(x,\ell) = \{y \in \T_r^d : d_{\rm int}(x,y) \le \ell \}.
\end{equation}
Thus
\begin{equation}
\{x \conn y \text{ by a path of length}\leq \ell\} = \{y\in B_\mathrm{int}(x,\ell)\}.
\end{equation}
We write the boundary of the intrinsic ball as
\begin{equation}
    \partial B_\mathrm{int}(x,\ell) =  B_\mathrm{int}(x,\ell)  \setminus B_\mathrm{int}(x,\ell-1)
    = \{y \in \T_r^d : d_{\rm int}(x,y) = \ell \}.
\end{equation}
Given a subset $g$ of
edges of the edge set $\B$ of $\T_r^d$ or $\Z^d$, we define $B^g_\mathrm{int}(x,\ell)$ similarly as $B_\mathrm{int}(x,\ell)$
except that the intrinsic distance from $x$ to $y$ is determined only using
paths consisting of edges of $g$.

Kozma and Nachmias \cite{KN09} computed the asymptotic
behaviour of the critical instrinsic one-arm probability in high dimensions to be
\begin{equation}
\label{e:1-arm-int}
\P_{p_c}(\partial B_\mathrm{int}(0,\ell) \neq \varnothing) \asymp \frac{1}{\ell}
\end{equation}
for every $\ell \geq 1$.
In fact, they also proved an extension of the upper bound of \eqref{e:1-arm-int} involving the intrinsic ball restricted
to a subgraph $g$; their proof also extends immediately to the torus as explained in \cite[Theorem~2.1(i)]{HS14}
and implies in particular that
\begin{equation}
  \label{e:bd-intrad-torus}
  \max_{g \subset \B(\T_r^d)}\P_{p_c}^\T(\partial B^{g}_\mathrm{int}(0,\ell)\neq \varnothing) \preceq \frac{1}{\ell}.
\end{equation}
for every $r,\ell\geq 1$.
(An important remark is that the proof of \eqref{e:bd-intrad-torus} does not
require the lace expansion on $\T_r^d$, but instead uses results for $\Z^d$ along with the coupling of torus and $\Z^d$ percolation---this
is discussed in greater detail in the verification of \cite[Theorem~4.1(b)]{HHII11}.
As such, there is no circular
reasoning here, nor an appeal to any result obtained via the torus lace
expansion.)

We first isolate two estimates in the following lemma, whose proof uses a standard coupling
of percolation at different values of $p$.

\begin{lemma}
For $0\leq p<q \leq 1$, $\ell \ge 1$, and $g \subset \B(\T_r^d)$,
\begin{align}
	\label{e:ub-intrad-pq}
    \P^\T_{q}(\partial B^g_\mathrm{int}(0,\ell) \neq \varnothing)
	&\leq \left(\frac{q}{p}\right)^\ell
    \P^\T_{p}(\partial B^g_\mathrm{int}(0,\ell) \neq \varnothing ),  \\
	\label{e:supercrit-coupling-2}
	\P_{q}^\T(x \in B_\mathrm{int}(0,\ell))
	&\leq \left(\frac{q}{p}\right)^\ell \P_{p}^\T(x \in B_\mathrm{int}(0,\ell)).
\end{align}
\end{lemma}

\begin{proof}
Let $0\leq p<q \leq 1$.  We begin with \eqref{e:ub-intrad-pq}.
Given a subset $g$ of the edge set of $\T_r^d$, we
write $R_g(\ell) =\{ \partial B^g_\mathrm{int}(0,\ell) \neq \varnothing \}$. We use the
standard
coupling of percolation configurations via uniform random variables assigned to each edge of the torus (see \cite[p.~11]{Grim99});
these uniform random variables are defined on some probability space $(\Omega,\mathcal{A}, \Q)$. We write $\eta^\T_p$ for the induced percolation configuration.
Since $R_g(\ell)$ depends on the edges inside $\T_r^d$ only, hence on finitely many edges, we can write
\begin{align}
	\Q(\eta_p^\T \in R_g(\ell), \eta_q^\T \in  R_g(\ell)) &= \sum_{\omega \in \{0,1\}^{\T_r^d}}
	\Q(\eta_p^\T \in R_g(\ell),\, \eta_q^\T = \omega,\, \omega \in R_g(\ell) ) \nnb
				&=\sum_{\omega \in R_g(\ell)}
				\Q(\eta^\T_p \in R_g(\ell)\mid \eta_q^\T = \omega)\Q(\eta_q^\T = \omega)
    .
\end{align}
Since the above left-hand side is at most
$\Q(\eta_p^\T \in R_g(\ell))$, \eqref{e:ub-intrad-pq} follows once we prove that
\begin{equation}
\label{e:qp}
    \Q(\eta^\T_p \in R_g(\ell)\mid \eta_q^\T = \omega) \ge \left(\frac{p}{q}\right)^\ell .
\end{equation}

To prove \eqref{e:qp}, we first observe that
on a specific configuration $\omega \in R_g(\ell)$,
there exists inside $\omega$
a deterministic path of open edges, starting from $0$, of
length $\ell$.
A fortiori, on the event
$\{\eta^\T_q = \omega\}$ there exist $\ell$ independent uniform
random variables $U_1,\cdots, U_\ell$ attached to these open edges such
that $U_i \leq q$ for all $1\leq i \leq \ell$.
For $\{\eta^\T_p \in R_g(\ell)\}$ to occur, it is
enough that $U_i \leq p$ for all $1\leq i \leq \ell$ which gives
\begin{align}
	\Q(\eta^\T_p \in R_g(\ell)\mid \eta_q^\T = \omega)
	&\geq
		\Q\Big(\bigcap\limits_{i = 1}^\ell \{U_i \leq   p
        \}\mid \eta_q^\T = \omega\Big)
		.
\end{align}
Since $U_1, \cdots, U_\ell$ are independent of the other uniform random variables,
the above right-hand side is equal to
\begin{align}
	\Q\Big(\bigcap\limits_{i = 1}^\ell \{U_i \leq p\}\mid \eta_q^\T = \omega\Big)
	&= \Q\Big(\bigcap\limits_{i = 1}^\ell \{U_i \leq p\}\mid \bigcap\limits_{i = 1}^\ell \{U_i \leq q \}\Big)
= \left(\frac{p}{q}\right)^\ell,
\end{align}
which proves \eqref{e:qp} and hence completes the proof of \eqref{e:ub-intrad-pq}.

The proof of \eqref{e:supercrit-coupling-2} is almost identical,
using the fact that on $\{x \in B_\mathrm{int}(0,\ell)\}$
there must exist a path of length less than or equal to $\ell$ connecting $0$ to $x$. Then
keeping the uniform random variables along this path
 open upon reducing
$q$ to $p$ gives the result.
\end{proof}

We now prove \eqref{e:plateaupc_upper} in the following proposition.

\begin{prop}
Let $d>6$ and suppose that \eqref{eq:twopointassumption} holds on $\Z^d$.
Fix $\delta \in (0,1]$.
For all
$A>0$
there exists a constant $C_A$ such that
\label{lem:ub-above-p_c}
\begin{equation}
\label{e::ub-above-p_c}
	\tau_{p}^\T(x) \leq
    (1+\delta)
    \tau_{p_c}(x) + C_{A} \delta^{-2} V^{-2/3}
\end{equation}
for all $r>2$, $x\in \T_r^d$ and $p \in (p_c,p_c+A V^{-1/3}]$.
\end{prop}

\begin{proof}
Let $x \in \T_r^d$, $A>0$ and $\delta >0$.
By the monotonicity of $\tau_{p}^\T(x)$ in $p$ and the independence of the upper bound on $p$,
it suffices to prove \eqref{e::ub-above-p_c} for $p=p_c+\eps$ with $\eps = A V^{-1/3}$.
We set $\gamma = \lfloor \delta/\eps \rfloor$
and begin with
the decomposition
\begin{align}
\label{e:supercrit_bound_split}
	\P_{p_c + \eps}^\T(0 \conn x)
	&=\P_{p_c + \eps}^\T(x \in  \partial B_\mathrm{int}(0,\ell) \text{ for some }\ell )\nnb
	&\le \P_{p_c + \eps}^\T(x \in  B_\mathrm{int}(0,3\gamma))+\P_{p_c + \eps}^\T(x \in \partial B_\mathrm{int}(0,\ell) \text{ for some } \ell \geq 3\gamma).
\end{align}
For the first term in \eqref{e:supercrit_bound_split}, we use \eqref{e:supercrit-coupling-2} and Proposition~\ref{prop:ub-at-p_c} to see that
\begin{align}
\label{e:term1}
	\P_{p_c + \eps}^\T(x \in  B_\mathrm{int}(0,3\gamma))
	&\leq \left(\frac{p_c+\eps}{p_c}\right)^{3\gamma}\P^\T_{p_c}(x \in  B_\mathrm{int}(0,3\gamma))\nnb
    &\leq e^{3\gamma\eps/p_c}\P^\T_{p_c}(0 \conn x) \nnb
	&\leq e^{3\delta/p_c}(\tau_{p_c}(x) + C V^{-2/3}).
\end{align}

We consider now the second term in \eqref{e:supercrit_bound_split}.
For this we argue
as in the proof of \cite[(1.11)]{HS14} (see also \cite[Lemma~2.6]{KN09}).
We first observe that
on the event $\{ x \in \partial B_\mathrm{int}(0,\ell) \text{ for some } \ell \geq 3\gamma \}$,
the two events $\{\partial B_\mathrm{int}(0,\gamma)\neq \varnothing\}$
and $\{\partial B_\mathrm{int}(x,\gamma) \neq \varnothing\}$ must both occur.  However,
these two events do not necessarily occur disjointly. Indeed, in order for $d_{\text{int}}(0,y)$ to be large there
must not only be a long open path from $0$ to $y$ but also \emph{no shorter path}; the sets of closed edges guaranteeing that no such short path exists may be
shared in the common realisation of the two events. On the other hand,
the set of vertices in $B_\mathrm{int}(0,\gamma)$ and in $B_\mathrm{int}(x,\gamma)$ are disjoint.
To deal with this situation,
we define $G$ to be the random graph whose edge set consists of all
edges which touch $B_\mathrm{int}(x,\gamma-1)$ (the vertex set of $G$ consists
of the vertices incident to at least one edge in this set); these are exactly the edges needed to
determine the random set $B_\mathrm{int}(x,\gamma)$. From now on we identify subgraphs of the torus with subsets of $\bbB(\bbT_r^d)$ and write $g^c$ for the complement of a subgraph $g \subseteq \bbB(\bbT_r^d)$. Since
$B_\mathrm{int}(0,\gamma)$ and $B_\mathrm{int}(x,\gamma)$ are disjoint, we see that
\begin{align}
	& \P_{p_c + \eps}^\T( x \in \partial B_\mathrm{int}(0,\ell) \text{ for some } \ell \geq 3\gamma )
	\nnb &\qquad \leq \P_{p_c + \eps}^\T(\partial B^{G^c}_\mathrm{int}(0,\gamma)\neq \varnothing,\;\partial B_\mathrm{int}(x,\gamma) \neq \varnothing)\nnb
	&\qquad
= \sum_{g \subset \B(\T_r^d)}\P_{p_c + \eps}^\T(\partial B^{g}_\mathrm{int}(0,\gamma)\neq \varnothing,\;\partial B_\mathrm{int}(x,\gamma) \neq \varnothing, \; G^c = g).
\end{align}
By definition of $G$, the events $\{\partial B^{g}_\mathrm{int}(0,\gamma)\neq \varnothing \}$
and $\{ G^c = g,\;\partial B_\mathrm{int}(x,\gamma) \neq \varnothing\}$ depend on different edges (namely those of $g$ and $g^c$ respectively), and hence
\begin{align}
	&\P_{p_c + \eps}^\T( x \in \partial B_\mathrm{int}(0,\ell) \text{ for some } \ell \geq 3\gamma )
	\nnb & \qquad
\leq  \sum_{g \subset \B(\T_r^d)}\P_{p_c + \eps}^\T(\partial B^{g}_\mathrm{int}(0,\gamma)\neq \varnothing)\P_{p_c + \eps}^\T(\partial B_\mathrm{int}(x,\gamma) \neq \varnothing, \; G^c = g) \nnb
	&\qquad \leq \max\limits_{g \subset \B(\T_r^d)}\P_{p_c + \eps}^\T(\partial B^{g}_\mathrm{int}(0,\gamma)\neq \varnothing) \P_{p_c + \eps}^\T(\partial B_\mathrm{int}(x,\gamma) \neq \varnothing)\nnb
\label{e:longsuper}
	&\qquad \leq \max\limits_{g \subset \B(\T_r^d)}\P_{p_c + \eps}^\T(\partial B^{g}_\mathrm{int}(0,\gamma)\neq \varnothing)^2 .
\end{align}
Now, we have by \eqref{e:ub-intrad-pq} with $p=p_c$ and $q = p_c+\eps$
 and by \eqref{e:bd-intrad-torus} that
\begin{equation}
    \P_{p_c+\eps}^\T(\partial B^{g}_\mathrm{int}(0,\gamma)\neq \varnothing)
    \le
    \left(\frac{p_c+\eps}{p_c}\right)^{\gamma}
    \P_{p_c}^\T(\partial B^{g}_\mathrm{int}(0,\gamma)\neq \varnothing) \preceq  \left(\frac{p_c+\eps}{p_c}\right)^{\gamma}
   \frac{1}{\gamma}
\end{equation}
and hence that
\begin{align}
\label{e:term2}
	\P_{p_c + \eps}^\T(x \in \partial B_\mathrm{int}(0,\ell) \text{ for some } \ell \geq 3\gamma)
	&\preceq \frac{1}{\gamma^2} \left(\frac{p_c+\eps}{p_c}\right)^{2\gamma}
    \preceq \frac{\eps^2}{\delta^2} e^{2\delta/p_c} .
\end{align}

Altogether, by \eqref{e:term1} and \eqref{e:term2} we therefore have
\begin{equation}
	\tau_p^\T(x) \leq e^{3\delta/p_c}\tau_{p_c}(x)
	+C\delta^{-2}e^{3\delta/p_c}V^{-2/3}.
\end{equation}
Finally, we replace
$e^{3\delta/p_c}$ by $1+\delta'$,
and we may obtain in this way
any $\delta'\in (0,1]$.
This gives the desired result and the proof
is complete.
\end{proof}

\section*{Acknowledgements}
This work was carried out primarily while TH was a Senior Research Associate at the University of Cambridge, during which time he was supported by ERC starting grant 804166 (SPRS).
The work of EM and GS was supported in part by NSERC of Canada.


\end{document}